\documentclass[11pt,twoside]{article}
\usepackage[a4paper,innermargin=1.2in,outermargin=1.2in,
bottom=1.5in,marginparwidth=1in,marginparsep
=3mm]{geometry}

\usepackage{amssymb}
\usepackage{microtype}
\usepackage[shortlabels]{enumitem}

\usepackage{secdot}

\usepackage{amsmath, amsthm, bm}
\usepackage{bbm}
\usepackage{setspace}
\usepackage[
bookmarks=true,
bookmarksnumbered=true,
colorlinks=true, pdfstartview=FitV, linkcolor=blue, citecolor=blue,
urlcolor=blue]{hyperref}
\usepackage[nameinlink,capitalize]{cleveref}
\usepackage[shortlabels]{enumitem}
\usepackage[dvipsnames]{xcolor}
\usepackage{comment}
\usepackage{subcaption}
\usepackage[english]{babel}

\newtheorem{theorem}{Theorem}[section]
\newtheorem{lemma}[theorem]{Lemma}
\newtheorem{proposition}[theorem]{Proposition}
\newtheorem{corollary}[theorem]{Corollary}
\newtheorem*{lemma*}{Lemma}

\theoremstyle{definition}
\newtheorem{definition}[theorem]{Definition}
\newtheorem{assumption}[theorem]{Assumption}

\newtheorem{remark}[theorem]{Remark}
\numberwithin{equation}{section}

\newcommand{\A}{\mathcal{A}}
\newcommand{\B}{\mathcal{B}}
\newcommand{\C}{\mathcal{C}}

\newcommand{\D}{\mathbb{D}}
\newcommand{\DD}{\mathcal{D}}
\newcommand{\Di}[1]{|#1|}
\newcommand\E{\hskip.15ex\mathsf{E}\hskip.10ex}
\newcommand{\F}{\mathcal{F}}
\newcommand{\FF}{\mathbb{F}}
\DeclareMathOperator{\I}{\mathbbm{1}}
\newcommand{\N}{\mathbb{N}}
\renewcommand\P{\mathsf{P}}
\newcommand{\R}{\mathbb{R}}

\newcommand{\Z}{\mathbb{Z}}
\newcommand{\1}{\mathbbm{1}}
\newcommand{\eps}{\varepsilon}
\newcommand{\wt}{\widetilde}
\newcommand{\diff}{\,\mathrm{d}}
\newcommand{\var}{\mathrm{var}}
\DeclareMathOperator{\Var}{Var}
\DeclareMathOperator{\sign}{\mathrm{sign}}

\DeclareMathOperator{\Law}{\mathrm{Law}}

\newcommand{\nn}{\nonumber}

\newcommand*{\Gref}[1]{\hyperref[SDEfunc]{Eq($#1$)}}
\newcommand*{\Gtag}[1]{\tag{Eq($#1$)}}

\title{Weak existence for SDEs with singular drifts and fractional Brownian or L\'evy noise beyond the subcritical regime}
\author{Oleg Butkovsky, Samuel Gallay}

\author{
	Oleg Butkovsky%
	\thanks{Weierstrass Institute, Mohrenstrasse 39, 10117 Berlin, FRG. Email: \texttt{oleg.butkovskiy@gmail.com}}
	\and
	\setcounter{footnote}{3}
	Samuel Gallay%
\thanks{Université Côte-d'Azur, LJAD, Nice, France. Email: \texttt{samuel.gallay@univ-cotedazur.fr}}
	}

\date{February 10, 2026}

\begin{document}

\maketitle
\abstract{We study a multidimensional stochastic differential equation with additive noise:
\[
d X_t=b(t, X_t) dt +d \xi_t,
\]
where the drift $b$ is integrable in space and time, and $\xi$ is either a fractional Brownian motion or a Lévy process. We show weak existence of solutions to this equation under the optimal condition on integrability indices of $b$, going beyond the subcritical Krylov--Röckner (Prodi--Serrin--Ladyzhenskaya) regime. This extends the recent results of Krylov (2020) to the fractional Brownian and Lévy cases. We also construct a counterexample to demonstrate the optimality of this condition. In the one-dimensional case, we show the existence of a strong solution under the same condition. Our methods are built upon 
a version of the stochastic sewing lemma of L\^e and the John--Nirenberg inequality.	}



 \section{Introduction}

We study the existence of weak solutions to stochastic differential equations (SDEs) driven by fractional Brownian or Lévy noise:
\begin{align}
  \label{eq:main_eq} X_t &= x + \int_0^t b_r(X_r)\diff r + W_t^H, \quad \quad \quad x\in \R^d,\, t\in[0,T],\\
  \label{eq:main_eq_levy}
  X_t &= x + \int_0^t b_r(X_r)\diff r + L_t, \quad \quad \quad x\in \R^d,\, t\in[0,T],
\end{align}
where $d\in\N$, $T>0$, $b\colon [0, T] \times \R^d \to \R^d$ is a function in the mixed Lebesgue space $L_q([0, T], L_p(\R^d))$, $p \in [1, \infty)$, $q\in[1,\infty]$. The noise $W^H$ is a fractional Brownian motion (fBm) of Hurst parameter $H \in (0, 1)$, and $L$ is a Lévy process. We show that if 
\begin{equation}
	\label{eq:main_cond} \frac{1-H}{q}+\frac{Hd}{p}< 1 - H,
\end{equation}
then the equation~\eqref{eq:main_eq} has a weak solution (Theorem~\ref{th:weak_existence}) and for $d=1$ equation~\eqref{eq:main_eq} has a strong solution (Theorem~\ref{th:strong_existence}). We also construct a counter-example showing the optimality of this condition (Theorem~\ref{th:counter_example}). We also show that if the process $L$ is  an $\alpha$-stable process or its ``relative'' (see Theorem~\ref{th:levy_existence} for the precise condition), $\alpha\in(1,2)$, and 
\begin{equation}
	\label{eq:main_cond_levy} \frac{\alpha-1}{q}+\frac{d}{ p}< \alpha - 1,
\end{equation}
then the equation~\eqref{eq:main_eq_levy} has a weak solution (Theorem~\ref{th:levy_existence}) and for $d=1$ it has a strong solution (Theorem~\ref{th:str_levy_existence}). This extends the recent results by Krylov \cite{Kr20}, where these statements were proved in the Brownian setting ($H=\frac12$).

Let us recall that deterministic ordinary differential equations can be ill-posed, meaning that they may have multiple or no solutions, but by adding noise, they can become well-posed. This phenomenon is called \textit{regularization by noise}.

In the case where the noise is a standard Brownian motion, the problem has been widely studied. The first results in this direction are due to Zvonkin \cite{Zvo74} and Veretennikov \cite{Ver80}, who proved the strong existence and uniqueness of solutions to the SDE
\begin{equation}\label{bmeq}
dX_t=b(t,X_t) dt +dW_t,
\end{equation}
where $b$ is a bounded measurable function and $W$ is a Brownian motion. Later, Krylov and Röckner  \cite{KR05} and Zhang \cite{Zhang11} extended these results to the case where $b\in L_q([0, T], L_p(\R^d))$, $p,q\in[1,\infty]$ and
\begin{equation} \label{eq:KR} 
	\frac{2}{q} + \frac{d}{p} < 1.
\end{equation}
This condition is often referred to in the literature as the Ladyzhenskaya--Prodi--Serrin condition (LPS). One can observe that if~\eqref{eq:KR} is satisfied, then at small scales, the noise dominates the drift. Furthermore, the noise is less regular than the drift, which allows for regularization by noise to occur. This regime is often called subcritical. For a more detailed discussion, we refer to \cite[Section~1.1]{GG23}.

If, on the contrary, $\frac{2}{q} + \frac{d}{p} > 1$ (supercritical regime), then at small scales, the drift dominates, and as a result, regularization by noise is not guaranteed. In fact, for any values of $p, q, d$ satisfying this inequality, there exists a drift $b \in L_q([0, T], L_p(\R^d))$ for which even weak uniqueness fails for SDE~\eqref{bmeq}, see  \cite[Section~1.3]{GG23}.

The weak existence is more nuanced. If the time integrability is $q=\infty$, then for any values of $p,d$ in the  supercritical regime, there exists a drift $b\in L_\infty([0, T], L_p(\R^d))$ for which the weak existence 
fails for SDE~\eqref{bmeq}, see, e.g., \cite[Example~2.1]{Sasha}. Surprisingly enough, if $q\neq\infty$, then the weak existence does hold for some supercritical values of $p,q,d$. Namely, it was discovered by Krylov \cite{Kr20} (see \cite{galeati2023note} for an alternative proof) that a weaker condition
\begin{equation} \label{eq:KRweak} 
	\frac{1}{q} + \frac{d}{p} \le 1
\end{equation}
implies the weak existence for the SDE~\eqref{bmeq} with any drift $b\in L_q([0, T], L_p(\R^d))$ for  $p\ge q$, and $b\in  L_p(\R^d,L_q([0, T]))$ for  $p\le q$. It was also shown in \cite{Kr20} that in the regime $p\le q d$ this condition is optimal: if $p,q,d$ satisfy 	$\frac{1}{q} + \frac{d}{p} > 1$ and $p\le q d$, then there exists a drift $b\in L_q([0, T], L_p(\R^d))$ for which no weak solution exists for SDE~\eqref{bmeq}. The construction of a counter-example for the remaining values of parameters (that is, $p> q d$, 	$\frac{1}{q} + \frac{d}{p} > 1$) was left open in \cite{Kr20}.

Thus, if $q=\infty$, then Krylov's condition~\eqref{eq:KRweak} reduces to the LPS condition~\eqref{eq:KR}, and a weak solution exists only for values of $p,d$ that are in the subcritical regime. If $q<\infty$, then~\eqref{eq:KRweak} is weaker than~\eqref{eq:KR}, and a weak solution also exists for some supercritical $p,d$.
 Note, however, that under additional assumptions on the drift (e.g., that the drift is divergence-free), Krylov's condition~\eqref{eq:KRweak} can be further relaxed, see \cite{ZZ} and \cite[Section~1]{hao2023sdes} for a very detailed discussion.

Much less is known for the case where the noise is a fractional Brownian motion  $W^H$. The main reason is that for $H \neq \frac12$, the fractional Brownian motion is neither a martingale nor a Markov process. Since the techniques used in \cite{Zvo74, Ver80, KR05, Kr20} for both proving well-posedness and constructing counterexamples are based on Itô's formula, they become inapplicable.

A breakthrough in understanding the well-posedness of SDEs driven by fBm was achieved in  \cite{CG} and \cite{LeSSL}, where sewing and stochastic sewing techniques were developed. Extending their ideas,  \cite{GG23} showed (among many other interesting results), the weak existence for SDE~\eqref{eq:main_eq} if $b\in L_q([0, T], L_p(\R^d))$, $p,q\in[1,\infty]$, and 
\begin{equation}
	\label{eq:GG23} \frac{1}{q} + \frac{Hd}{p} < 1-H \quad \text{and} \quad \frac{2Hd}{p}<1-H.
\end{equation}
When $H=\frac{1}{2}$, the first part of~\eqref{eq:GG23} reduces to the LPS condition~\eqref{eq:KR} and the second part becomes an additional condition, $2d < p$. Thus, one does not recover the optimal Krylov's condition~\eqref{eq:KRweak}.

We improve this result in two directions. First, we show that the extra assumption $\frac{2Hd}{p}<1-H$ can be dropped. Then, we demonstrate that the condition $\frac{1}{q} + \frac{Hd}{p} < 1-H$ can be relaxed to~\eqref{eq:main_cond}. We note that~\eqref{eq:main_cond} can be viewed as an extended Krylov's condition because it reduces to~\eqref{eq:KRweak} when $H = \frac12$.
This is the subject of Theorem~\ref{th:weak_existence}. We note that the time-independent case was studied in \cite{BLM23}. Theorem~\ref{th:counter_example} shows the optimality of the condition~\eqref{eq:main_cond} by constructing the corresponding counter-example. Theorem~\ref{th:counter_example} seems to be new even in the Brownian case $H=1/2$ as the case $p>qd$ has not been covered in  \cite{Kr20}.

Theorem~\ref{th:levy_existence} establishes weak existence for SDEs driven by a L\'evy process from a quite general class ($\alpha$-stable-like processes), under condition~\eqref{eq:main_cond_levy}. This condition is an analog of Krylov's condition for L\'evy processes. 
Previously, weak existence for~\eqref{eq:main_eq_levy} was proved in \cite[Theorem~4.1]{Zhang13} for $\alpha \in (1, 2)$ under the condition
\begin{equation*}
\frac{\alpha }{q}+\frac{d}{p} < \alpha -1,
\end{equation*}
with the autonomous case having been treated earlier by \cite{Portenko}. Weak uniqueness in the same regime was later established in \cite{Jin18}. This condition is more restrictive than \eqref{eq:main_cond_levy} and coincides with~\eqref{eq:KR} in the Brownian case. Note that all the aforementioned works \cite{Zhang13,Portenko,Jin18} considered only the case where the noise is a symmetric $\alpha$-stable process. Well-posedness for SDEs  driven by $\alpha$-stable-like noises  with more regular drifts (with H\"older drifts) was obtained in \cite{CZZ}.

There are also many recent results \cite{CdRM22,LZ22,KP22, KP25rough} tackling  the existence and uniqueness of solutions to \eqref{eq:main_eq_levy} when the drift is not necessarily a function but only a distribution. For example, in \cite[Theorems 2.12 and 3.12]{KP25rough}, the authors introduce a new notion of rough weak solutions, equivalent to solving the martingale problem, and show well-posedness. In the case of autonomous drifts, they require the drift to be in $\C^{\beta }$ with $\beta \in (\frac{2-2 \alpha }{3}, 0)$, which translates to $\frac{d}{p} < \frac{2}{3}(\alpha -1)$ if $b\in L_p(\R^d)$. While this is a very significant improvement over the usual Young condition $\beta \in (\frac{1-\alpha }{2}, 0)$ for generic drifts, in the case of integrable drifts this condition is still more restrictive than $\frac{d}{p} < \alpha -1$ from \cite{Portenko}.

To obtain these results, we used stochastic sewing techniques, the John--Nirenberg inequality, and taming singularities methods, building on \cite{BLM23}, where these results were established for the time-independent case. We would like to point out that while our general strategy  follows that of \cite{BLM23}, a direct application would lead to a suboptimal condition, $\frac{1}{q} + \frac{Hd}{p} < 1-H$. To derive a less restrictive condition~\eqref{eq:main_cond}, we had to conduct a much more intricate analysis and obtain improved integral bounds, see, in particular, Lemma~\ref{lem:step02a}.

\smallskip

The rest of the article is organized as follows. We present our main results in Section~\ref{s:mr}, and the proofs are provided in Section~\ref{sec:main_result}. All auxiliary technical results are presented in Appendix.

\smallskip
\textbf{Convention on constants}. Throughout the paper, $C$ denotes a positive constant whose value may change from line to line; its dependence is always specified in the corresponding statement.

\smallskip
\textbf{Acknowledgements}. The authors are grateful to Lucio Galeati, Máté Gerencsér, and Leonid Mytnik for  very helpful discussions.  We would  like to thank the referees for their constructive feedback and many useful suggestions, which have helped us improve the manuscript. OB has received funding from the Deutsche Forschungsgemeinschaft (DFG, German Research Foundation) under Germany's Excellence Strategy --- The Berlin Mathematics Research Center MATH+ (EXC-2046/1, project ID: 390685689, sub-project EF1-22) and DFG CRC/TRR 388 ``Rough Analysis, Stochastic Dynamics and Related Fields'', Project B08. The main part of the research was done while SG was doing an internship at the Weierstrass Institute in May--July 2023. Part of the work on the project was done during the visit of the authors to Università degli Studi di Torino. We would like to thank these institutions for providing excellent working conditions, support, and hospitality.

\section{Main results}\label{s:mr}

We begin by introducing the necessary notation and recalling the standard definitions. Let $d\in\N$, $T>0$. Denote by  $\C([0,T],\R^d)$ be the space of continuous functions $[0,T]\to\R^d$ equipped with the usual supremum norm. Let  $(\Omega, \F, \P)$ be a probability space and let $\FF = (\F_{t})_{t\in[0,T]}$ be a filtration.
It is well-known that if $W^{H}$ is a $d$-dimensional fractional Brownian motion on this space, then there exists a $d$-dimensional Brownian motion $B$ such that the following representation holds:
\begin{equation}\label{fbmr}
  W_{t}^{H} = \int_{0}^{t} K_{H}(t, s)\diff B_{s}, \quad t \in [0, T],
\end{equation}
where for $H>1/2$ the kernel $K_H$ is given by 
\begin{equation}\label{fbmr1}
K_H(t,s):=C(H,d)s^{\frac12-H}\int_s^t (r-s)^{H-\frac32}r^{H-\frac12}\,dr,
\end{equation}
and for $H\le1/2$ 
\begin{equation}\label{fbmr2}
K_H(t,s):=C(H,d)\Bigl(t^{H-\frac12}s^{\frac12-H}(t-s)^{H-\frac12}  
+(\frac12-H)s^{\frac12-H}\int_s^t (r-s)^{H-\frac12}r^{H-\frac32}\,dr\Bigr),
\end{equation}
for some constant $C=C(H,d)>0$.
We say that $W^H$ is a \textit{$\FF$-fractional Brownian motion} if there exists an $\FF$-Brownian motion $B$ such that~\eqref{fbmr} holds.

\textit{A weak solution} to the SDE~\eqref{eq:main_eq} is a pair $(X, W^{H})$ on a complete filtered probability space $(\Omega, \F, \FF, \P)$, such that $W^{H}$ is a $\FF$-fractional Brownian motion, $X$ is continuous, adapted to $\FF$, and $\P\bigl(\text{\eqref{eq:main_eq} holds for all $t\in[0,T]$}\bigr)=1$.

Now we are ready to present our first main result.
\begin{theorem}[Existence of weak solutions: fractional Brownian noise]
	\label{th:weak_existence}
	Let $d\in\N$, $H \in (0, 1)$, $p \in [1, \infty )$, $q \in [1, \infty]$, $x \in \R^d$, $T>0$, $b\in L_q([0, T], L_p(\R^d))$. Suppose that the condition~\eqref{eq:main_cond} is satisfied. Then the following holds.
	\begin{enumerate}[\rm{(}i\rm{)}]
	\item The SDE~\eqref{eq:main_eq} has a weak solution.
	\item 
Let $\{b^n, n \in \N \}$ be a sequence of smooth bounded functions converging to $b$ in $L_q([0, T], L_p(\R^d))$ as $n\to\infty$. Assume that the sequence $\{x^n, n \in \N\}$, where $x^n\in\R^d$, converges to $x$ as $n\to\infty$. Let $X^n$ be the strong solution to
	\begin{equation}
		\label{eq:approx} X^n_t = x^n + \int_0^t b_r^n(X_r^n)\diff r+ W_t^{H}, \quad t\in[0, T].
	\end{equation}
	Then the sequence $(X^n, W^{H})_{n\in\N}$ is tight in $C([0,T],\R^{2d})$, and any of its partial limits is a weak solution to the SDE~\eqref{eq:main_eq}.
	\end{enumerate}
\end{theorem}
Theorem~\ref{th:weak_existence} extends the corresponding results from \cite[Section~8]{GG23}.
If $q=\infty$, this statement was established in \cite[Theorem~2.6]{BLM23}. As mentioned before, in this case, the condition~\eqref{eq:main_cond} reduces to the condition of subcriticality. On the other hand, if $q\neq\infty$, then the values of $H,d,p$ might belong to the supercritical regime. 
 If  $H=\frac12$, then Theorem~\ref{th:weak_existence} provides an alternative proof of \cite[Theorem~3.1]{Kr20}.

The following counter-example shows that the condition~\eqref{eq:main_cond} in Theorem~\ref{th:weak_existence} is optimal, except perhaps for the equality case in~\eqref{eq:main_cond}, which is not treated. This result generalizes Krylov's counter-example \cite[Example 2.1]{Kr20} to the general case $H \neq 1/2$, although our methods of proof are very different. Note that in \cite[Example 2.1]{Kr20}, it was assumed additionally that $p\le qd$. Theorem~\ref{th:counter_example} shows that this assumption is not needed.

\begin{theorem}
	\label{th:counter_example}
	Let $d \in\N$, $H\in(0, 1)$, and $p, q \in[ 1,\infty)$  be such that
	\begin{equation}
		\label{pqnosol} \frac{1-H}{q} + \frac{Hd}{p} > 1 -H.
	\end{equation}
	Then there exists $b\in L_q([0, 1], L_p(\R^d))$ such that the equation~\eqref{eq:main_eq} has no weak solution.
\end{theorem}

In the one-dimensional case, we can show strong existence of solutions to \eqref{eq:main_eq} under the same conditions.

Recall that a \emph{strong solution} to the SDE~\eqref{eq:main_eq} is a pair $(X, W^{H})$ on a complete filtered probability space $(\Omega, \F, \P)$, such that $W^{H}$ is a fractional Brownian motion, $X$ is continuous, adapted to the completion of the filtration generated by $W^{H}$, and $\P\bigl(\text{\eqref{eq:main_eq} holds for all $t\in[0,T]$}\bigr)=1$.

\begin{theorem}[Existence of strong solutions: fractional Brownian noise]
\label{th:strong_existence}
Let $d=1$, $H \in (0, 1)$, $p \in [1, \infty )$, $q \in [1, \infty]$, $x \in \R$, $T>0$, $b\in L_q([0, T], L_p(\R))$. Suppose that the condition~\eqref{eq:main_cond} is satisfied. Then SDE~\eqref{eq:main_eq} has a strong solution.
\end{theorem}

Now, we move to the results concerning the weak existence for the SDE~\eqref{eq:main_eq_levy}. We denote by $\D([0,T],\R^d)$ the Skorokhod space of càdlàg functions $[0,T]\to\R^d$. A pair $(X, L)$ on a complete filtered probability space $(\Omega, \F, \FF, \P)$ is a \textit{weak solution} to~\eqref{eq:main_eq_levy} if $X$ is càdlàg and adapted to $\FF$, $L$ is a $\FF$–Lévy process, and $\P\bigl(\text{\eqref{eq:main_eq_levy} holds for all $t\in[0,T]$}\bigr)=1$. 

Denote by $\Phi : \R^{d}\to \R$ the characteristic exponent of $L$, that is,
\begin{equation*}
\E e^{i \langle \lambda, L_t\rangle }=:e^{-t \Phi(\lambda)},\quad t\ge0,\,\,\lambda\in\R^d,
\end{equation*}

We work under quite generic assumptions on the driving L\'evy process $L$. We assume that for some $\alpha \in (1, 2)$, $c_{1}, c_{2}, N \in (0, \infty )$ we have 
\begin{equation}
\label{eq:characteristic_bound}
c_{1} |\lambda |^{\alpha} \le \Re \Phi (\lambda)\le c_{2}|\lambda |^{\alpha }, \quad\text{when $|\lambda|\ge N$}.
\end{equation}
This is a standard assumption in the field, and it is satisfied, for example, if $L$ is a symmetric $\alpha$-stable process (in which case $\Phi(\lambda)=c_\alpha|\lambda|^{\alpha}$ for some $c_\alpha>0$), a cylindrical $\alpha$-stable process ($\Phi(\lambda)=c_\alpha\sum_{i=1}^d|\lambda_i|^{\alpha}$), a relativistic $\alpha$-stable process, a truncated $\alpha$-stable process, and so on; see \cite[Examples 2.10–2.18]{BDG24}.

\begin{theorem}[Existence of weak solutions: Lévy noise]
	\label{th:levy_existence}
	Let $\alpha \in (1, 2)$, $p \in [1, \infty )$, $q \in [1, \infty]$, $x \in \R^d$, $T>0$, $b\in L_q([0, T], L_p(\R^d))$. Let $L$ be a Lévy process satisfying~\eqref{eq:characteristic_bound} and suppose that the condition~\eqref{eq:main_cond_levy} is satisfied. Then, the following holds.
	\begin{enumerate}[\rm{(}i\rm{)}]
		\item The SDE~\eqref{eq:main_eq_levy} has a weak solution.
		\item 
		Let $\{b^n,\ n \in \N \} $ be a sequence of smooth bounded functions converging to $b$ in $L_q([0, T], L_p(\R^d))$ as $n\to\infty$. Assume that the sequence $\{x^n, n \in \N\}$, where $x^n\in\R^d$, converges to $x$ as $n\to\infty$. Let $X^n$ be the  strong solution to
		\begin{equation}
			\label{eq:approx_levy} X^n_t = x^n + \int_0^t b_r^n(X_r^n)\diff r+L_t, \quad t\in[0, T].
		\end{equation}
		Then the sequence $(X^n, L)_{n\in\N}$ is tight in $\D([0, T], \R^{2d})$,
		and any of its partial limits is a weak solution to the SDE~\eqref{eq:main_eq_levy}.
	\end{enumerate}
\end{theorem}

To the best of our knowledge, this is the first result concerning the weak existence for SDEs driven by Lévy noise in the supercritical regime. If $q=\infty$, then~\eqref{eq:main_cond_levy} reduces to the subcritical condition, and this result was obtained in \cite{Portenko} for symmetric $\alpha$-stable processes. For general $p,q,d,\alpha$ and  symmetric $\alpha$-stable noise, 
the weak existence of solutions to~\eqref{eq:main_eq_levy} was established in \cite[Theorem~4.1]{Zhang13} under the subcritical condition $\alpha/q+d/p<\alpha-1$, which is more restrictive than~\eqref{eq:main_cond_levy}. We are not aware of any previous weak existence results for SDEs of type~\eqref{eq:main_eq_levy} with non-H\"older drift driven by generic L\'evy noise. Weak existence for H\"older drifts and generic Lévy noise was obtained in \cite[Theorem~1.1]{CZZ}.

Note that in the case $d=1$, $q=\infty$, condition~\eqref{eq:main_cond_levy} is optimal; see \cite[Section~9]{MW24}. We conjecture that this condition is optimal for the full range of parameters $d,p,q,\alpha$, though this remains an open problem.

As in the fractional Brownian case, the one-dimensional setting ($d = 1$) is special: in this case, strong existence holds.

\begin{theorem}[Existence of strong solutions: Lévy noise]
	\label{th:str_levy_existence}
	Let $d=1$,  $\alpha \in (1, 2)$, $p \in [1, \infty )$, $q \in [1, \infty]$, $x \in \R$, $T>0$, $b\in L_q([0, T], L_p(\R))$. Let $L$ be a Lévy process satisfying~\eqref{eq:characteristic_bound} and suppose that the condition~\eqref{eq:main_cond_levy} is satisfied. Then SDE~\eqref{eq:main_eq_levy} has a strong solution.
\end{theorem}

We recall that for $d=1$ strong existence of solutions to SDE \eqref{eq:main_eq_levy} driven by a symmetric $\alpha$-stable process, $\alpha \in (1, 2)$, was established in \cite{ABM} for $b\in\C^\beta$, $\beta>\frac12-\frac\alpha2$. If the drift $b\in L_p(\R)$ this condition becomes $\frac1p<\frac\alpha2-\frac12$ which is worse than our condition~\eqref{eq:main_cond_levy},  which is $\frac1p<\alpha-1$.

Let us briefly sketch our proof strategy. The proof of the counterexample relies on deterministic arguments and is presented in Section~\ref{sec:nwe}. A central component in establishing weak existence is a moment bound for additive functionals of the driving noise perturbed by a drift of finite $1$-variation (a Krylov-type estimate; see Lemma~\ref{lem:step02a}). This bound is derived using a recent refinement of the stochastic sewing method (the Rosenthal–Burkholder stochastic sewing lemma) together with the John–Nirenberg inequality. It is then used to obtain a priori bounds on the $1$-variation of the solutions (Corollary~\ref{cor:step02b}), which, in turn, imply tightness. The final step is to establish stability, which also relies on the moment bounds for additive functionals (Lemma~\ref{lem:stability}). Strong existence for $d = 1$ is shown using a version of the Gy\"{o}ngy–Pardoux argument \cite{GP93a, GP93b}.

\section{Proof of the main results} \label{sec:main_result}

Let us introduce further notation. For $\gamma\in(0,1)$ and a Borel subset $Q$ of $\R^k$, $k\in\N$, we denote by $\C^\gamma(Q,\R^d)$ the space of all $\gamma$-Hölder continuous functions $f\colon Q\to\R^d$.
Let $\B^{\gamma}_\rho=\B^{\gamma}_{\rho,\infty}$, where $\gamma\in\R$ and $\rho\in[1,\infty]$, be the Besov space of regularity $\gamma$ and integrability $\rho$. Throughout the proofs, we will use 
that for any Schwartz distribution $f$, $\rho \in [1, \infty]$, $\gamma<0$, $\lambda \in [0, 1]$ and $a, b \in \R$ one has for some $C=C(\gamma, \rho,\lambda)$
\begin{align}\label{besovpr}
&\| f(a + \cdot) \|_{\B_\rho^{\gamma}}= \| f \|_{\B_\rho^{\gamma}} \quad \text{and}\quad \| f(a + \cdot) - f(b + \cdot)\|_{\B_\rho^{\gamma}} \le C |b-a|^{\lambda}\| f \|_{\B_\rho^{\gamma +\lambda}};\\
&\|f(\lambda \cdot)\|_{\B_\rho^{\gamma}}\le \lambda^{\gamma-\frac{d}{\rho}} \|f\|_{\B_\rho^{\gamma}}
\label{besovscaling}.
\end{align}

For a function $f\colon [s,t]\to\R^d$, where $0\le s\le t$, denote its $1$-variation by
\[
[f]_{1-\var;[s,t]}:=\sup_\Pi\sum_{i=0}^{n-1} |f(t_{i+1})-f(t_i)|,
\]
where the supremum is taken over all the partitions $\Pi=\{t_0=s,t_1,\hdots, t_n=t\}$ of the interval $[s,t]$. The space of the functions $f\colon [s,t]\to\R^d$ with finite $1$-variation will be denoted by $\C^{1-\var}([s,t],\R^d)$.

For $0\le a \le b$, $n\in\N$, we denote  the simplex
\begin{equation}\label{simplex}
\Delta^n_{a, b}:=\{(t_1, \dots, t_n)\ \text{where}\ a \le t_1 \le \dots \le t_n \le b \}.
 \end{equation}

If a filtration $(\F_t)_{t\in[0,T]}$ is given, then we denote by $\E^t$ conditional expectation with respect to $\F_t$, $t\in[0,T]$. 

\subsection{Proof of the Theorem~\ref{th:counter_example}: no weak existence}\label{sec:nwe}

To show that under the condition~\eqref{pqnosol} the SDE~\eqref{eq:main_eq} has no solutions, we first begin with the following deterministic result.

\begin{lemma}
  \label{lem:deterministic_couter_example}
  Let $\alpha, \beta > 0$, $\gamma\in(0, 1)$ be such that
  \begin{equation}
    \label{eq:abnosol} \alpha \gamma + \beta (1-\gamma) > 1-\gamma.
  \end{equation}
  Let $f=(f^{1}, \dots, f^{d})\in\C^{\gamma}([0,1], \R^d)$ and suppose that $f(0) = 0$.  Then, the deterministic equation
  \begin{equation}
    \label{eq1} X_t^i = - \int_0^t \Big (\sign(X_s^i)|X_s|^{-\alpha}\1_{0<|X_{s}|<1} s^{-\beta} + \1_{|X_{s}|=0, s>0}s^{-2} \Big  )\diff s + f^{i}(t) 
 \end{equation}
 for $ t\in [0, 1]$ and $i=1,\hdots,d$, where the integral is understood as a \textbf{Lebesgue} integral, has no continuous solution $X=(X^1,\dots,X^d)\colon [0, 1] \to \R^d$.
\end{lemma}

The idea is that the first term of the drift $-\sign(X_s^i)|X_s|^{-\alpha}\1_{0<|X_{s}|<1} s^{-\beta}$ pushes $X$ towards $0$, while the additive term $f$, which in our particular case will be a path of a fractional Brownian motion, may push $X$ away from $0$. When~\eqref{eq:abnosol} is satisfied, the effect of the drift will be stronger, which will force the solution to stay at $0$, while second term of the drift $ \1_{|X_{s}|=0, s>0}s^{-2}$ prevents $0$ from being a solution to~\eqref{eq1}. 

\begin{proof}
Assume the contrary and let $X$ be a continuous solution to~\eqref{eq1} satisfying~\eqref{eq:abnosol}. We see that $X\equiv0$ cannot solve~\eqref{eq1} around $0$ because $s \mapsto s^{-2}$ is not integrable. Fix any time $t_{3}\in(0,1)$ such that $0<|X_{t_{3}}|<1$ (such $t_3$ exists, by continuity of $X$). With such a time $t_{3}$ in hand, define
  \begin{equation*}
    t_2 = \inf \{s>0: |X_s| \ge |X_{t_{3}}| \}.
  \end{equation*}
 Now take $i\in\{1,\hdots,d\}$ such that $|X_{t_2}^i| \ge |X_{t_2}^j|$ for every $1 \le j \le d$.
  With such a choice of $i$ we have $|X_{t_2}| \le d |X^i_{t_2}|$.
  Put now 
  \begin{equation*}
    t_1 = \sup \{s \le t_2: X^i_s = 0 \}.
\end{equation*}
We see that $t_1$ is well-defined since $X_0^i=0$ by assumption, and we note that $t_{1} \le t_{2} \le t_{3}$. Clearly, $X^i_{t_2}\neq0$, and we suppose without loss of generality that  $X^i_{t_2} > 0$. By the definition of $t_{1}$, it follows that $X^i_{s} > 0$ for all $s\in(t_1,t_2]$. By the definition of $t_{2}$ and $t_{3}$ we see that $|X_{s}| \le |X_{t_{3}}| <1$ for all $s \le t_{2}$. Therefore, using the continuity of $X$ and~\eqref{eq1}, we deduce
  \begin{equation*}
    |X_{t_3}| = |X_{t_2}| \le d|X_{t_2}^i| = d (X_{t_2}^i - X_{t_1}^i) = -d\int_{t_1}^{t_2} |X_s|^{-\alpha}s^{-\beta}\diff s + d(f^{i}(t_2) - f^{i}(t_1)).
\end{equation*}
  Now we use that $|X_s| \le |X_{t_2}| = |X_{t_{3}}|$ by definition of $t_2$ and that $f\in\C^{\gamma}$.
 We get for $K:=\|f\|_{\C^{\gamma}}$
   \begin{equation*}
    |X_{t_{3}}| \le -d|X_{t_{3}}|^{-\alpha}\int_{t_1}^{t_2} s^{-\beta}\diff s + dK(t_2-t_1)^{\gamma} \le -dt^{-\beta}_2|X_{t_{3}}|^{-\alpha}(t_2-t_1) + dK(t_2-t_1)^{\gamma}.
  \end{equation*}
  Taking the supremum over all values of $t_2 - t_1$ and noting that $t_2 \le  t_{3}$,
   \begin{equation*}
    |X_{t_{3}}| \le \sup_{u \ge 0} \bigl(-d t_{3}^{-\beta}|X_{t_{3}}|^{-\alpha}u + dKu^{\gamma} \bigr).
  \end{equation*}
  We compute the supremum by taking the derivative, and we find $u_{\text{max}}^{\gamma -1} = C t_{3}^{-\beta}|X_{t_{3}}|^{-\alpha}$ with $C = C (d, K, \gamma) > 0$. Therefore, by dominating the negative term by $0$ in the previous inequality, we obtain
   \begin{equation*}
    |X_{t_{3}}| \le C \bigl(|X_{t_{3}}|^{-\alpha}t_{3}^{-\beta}\bigr)^{\frac{\gamma}{\gamma -1}} = C |X_{t_{3}}|^{\frac{\alpha \gamma}{1-\gamma}}t_{3}^{\frac{\beta \gamma}{1-\gamma}}, \quad \text{so} \quad |X_{t_{3}}|^{1-\gamma -\alpha \gamma} \le C t_{3}^{\beta \gamma}.
  \end{equation*}
  Remember now that we can choose $t_{3}$ arbitrarily close to $0$, so if $1 - \gamma  - \alpha \gamma  \le  0$ then knowing that $|X_{0}| = 0$ we have a contradiction by continuity of $X$.
  Now, in the case $1 -\gamma -\alpha \gamma >0$, we find (since time $t_3$ was arbitrary and the constant $C$ does not depend on $t_3$)
  \begin{equation}
    \label{condsol} |X_t| \le C t^{\frac{\beta \gamma}{1-\gamma -\alpha \gamma}} \quad \text{for any $t \ge 0$}.
  \end{equation}
  For equation (\ref{eq1}) to be satisfied, we supposed that the Lebesgue integral is well-defined, and therefore 
  \[
    t \mapsto |X_t|^{-\alpha}\1_{0<|X_{t}|<1} t^{-\beta} + \1_{|X_{t}|=0}t^{-2} 
  \]
    is in $L_1([0, t_2])$.
  The function $t \mapsto t^{-1}$ is not integrable at $t=0$, thus for infinitely many $s_{i}$ close to $0$,
  \begin{equation}
    \label{condinteg} 
    |X_{s_{i}}|^{-\alpha}\1_{|X_{s_{i}}|>0} s_{i}^{-\beta} + \1_{|X_{s_i}|=0}s_{i}^{-2} 
\le t_{i}^{-1} 
  \end{equation}
  where we used again that $|X_r|<1$ for $r\in[0,t_2]$. 
If at one of such $s_{i}$, $X_{s_{i}}= 0$, then~\eqref{condinteg} implies $s_{i}^{-2} \le s_{i}^{-1}$ which is impossible. Therefore, $X_{s_{i}} \neq 0$ and we get
\begin{equation}
\label{lowerbound}
   |X_{s_{i}}|^{-\alpha} s_{i}^{-\beta} \le s_{i}^{-1}\quad \text{so} \quad |X_{s_{i}}| \ge s_{i}^{\frac{1-\beta }{\alpha }}.
\end{equation}
  Combining (\ref{condsol}) and (\ref{lowerbound}) and taking $s_{i} \to 0$ we get
   \begin{equation*}
    \frac{1-\beta}{\alpha} \ge \frac{\beta \gamma}{1-\gamma -\alpha \gamma}, \quad \text{which gives} \quad \beta (1-\gamma) + \alpha \gamma \le 1 - \gamma.
  \end{equation*}
  This contradicts (\ref{eq:abnosol}).
\end{proof}

\begin{proof}[Proof of Theorem~\ref{th:counter_example}]
  Because inequality~\eqref{pqnosol} is strict, there exist  $\alpha, \beta > 0$ and $\gamma \in (0, 1)$ such that equation~\eqref{eq:abnosol} is satisfied with
   \begin{equation*}
    \alpha <\frac{d}{p},\quad \beta < \frac{1}{q},\quad \text{and} \quad \gamma < H.
  \end{equation*}
  We define the function $b : [0, 1]\times \R^d \to \R^d$ by
   \begin{equation*}
    b_t^i(x) :=   
-\sign(x^i)|x|^{-\alpha} \1 _{|x|<1}t^{-\beta} -  \1_{|x|=0, t>0}t^{-2},\quad t\in[0,1],\,x\in\R^d.
  \end{equation*}
We see that $b\in L_q([0, 1], L_p(\R^d))$ because $\beta < \frac{1}{q}$ and  $\alpha < \frac{d}{p}$.
  A fractional Brownian motion $W^H$ is almost surely  in $\C^{\gamma}$ because $\gamma < H$.
  Applying Lemma~\ref{lem:deterministic_couter_example} with $f = W^H(\omega)$, this proves that equation (\ref{eq:main_eq}) with this choice of $b$ has almost surely no weak solution.
\end{proof}

\subsection{Recap of the sewing lemmas}
The proofs of Theorem~\ref{th:weak_existence} and Theorem~\ref{th:levy_existence} rely on the stochastic sewing lemma introduced in \cite{LeSSL}. For the convenience of the reader, we recall the corresponding statement as well as its modified version from \cite{BLM23}, which we also use.

Fix $0\le S\le T$ and the filtration $(\F_t)_{t\in[S,T]}$. 
Recall the notation~\eqref{simplex}. Let $(A_{s,t})_{(s,t)\in\Delta^2_{S,T}}$
be a collection of random elements taking values in $\R^d$, such that $A_{s,t}$ is $\F_t$-measurable whenever $(s,t)\in \Delta^2_{S,T}$. For any  $(s, u, t) \in \Delta^3_{S, T}$ denote as usual
\[
\delta A_{s,u,t}:=A_{s,t}-A_{s,u}-A_{u,t}.
\]

\begin{definition}\label{d:control}
 Let  $w : \Delta_{S, T}^2 \times \Omega \to \R_+$ be a measurable function. We say that it is a \textit{random control} if  for any sequence $(s_{n}, t_{n}) \in  \Delta_{S, T}^2$ such that $s_{n} \to s$ and $t_{n} \to t$, then
\[
  w(s_{n}, t_{n}) \to w(s, t) \quad \quad \text{almost surely},
\]
and if $w$ is superadditive, meaning that for any $(s, u, t) \in \Delta^3_{S, T}$ it satisfies  $w (s, u) +w (u, t) \le w (s, t)$ almost surely. 

 When $w$ does not depend on $\Omega$, we say that the control is deterministic.
\end{definition}

\begin{definition}
  \label{def:riemann_sums}
    We say that a random process $\A : \Omega \times [S, T] \to \R^d$ is a limit of Riemann sums of $\{A_{s, t}, (s, t) \in \Delta_{S, T}^2\}$ if, for every sequence of partitions $\Pi_n={\{s=t_0^n, t_1^n,\hdots, t_{k_n}^n=t\}}$, $n\in\Z_+$, of $[s, t] \subseteq [S, T]$ having a mesh size going to zero, we have
  \begin{equation}
    \label{eq:convergence} \sum_{i=0}^{k_n} A_{{t_i^n},{t_{i+1}^n}} \longrightarrow \A_t-\A_s\ \text{in probability as $n\to\infty$}.
  \end{equation}
\end{definition}

We begin with the stochastic sewing lemma of \cite{LeBanach} applied in the particular case where $\E^s\delta A_{s, u, t} = 0$ for all $(s, u,t) \in \Delta^3_{S, T}$.
\begin{proposition}[{\cite[Theorem~3.1 and Corollary~3.2]{LeBanach}}]\label{prop:simple_sewing}
  Let $\A : \Omega \times [S, T] \to \R^d$ be a random process that is a limit of Riemann sums of $(A_{s, t})$ as in Definition~\ref{def:riemann_sums}.
  Suppose that there exist $m \in [2, \infty)$, $\varepsilon > 0$, and a \textbf{deterministic} control $\omega$, such that for all $(s, u,t) \in \Delta^3_{S, T}$,
  \begin{enumerate}[$($a$)$]
    \item $\E^s\delta A_{s, u, t} = 0$; 
    \item $\| A_{s, t} \|_{L_m(\Omega)} \le\omega (s, t)^{\frac{1}{2}+\varepsilon}$.
  \end{enumerate}
  Then there exists a constant $C = C(m, \varepsilon)$, such that for every $ (s, t) \in \Delta^2_{S, T}$,
  \begin{equation*}
    \| \A_t -\A_s \|_{L_m(\Omega)} \le C \omega (s, t)^{\frac{1}{2}+\varepsilon}.
  \end{equation*}
\end{proposition}

The next statement is the Rosenthal--Burkholder stochastic sewing lemma. Its main advantage over the standard stochastic sewing lemma presented above is that to bound the $m$-th moment of $\A_t$, one only has to provide a very loose bound on high moments of $A_{s,t}$. The price to pay is that instead of controlling the second moments of $\delta A_{s,u,t}$, one now has to control the second \textit{conditional} moments of this quantity.

\begin{proposition}[{\cite[Theorem 3.6]{BLM23}}]
  \label{prop:rosenthal_sewing}
  Let $\A : \Omega \times [S, T] \to \R^d$ be a random process that is a limit of Riemann sums of $(A_{s, t})$ as in Definition~\ref{def:riemann_sums}.
  Suppose that there exist $\eps_0, \eps_1, \eps_2>0$ and random controls $\omega_1$ and $\omega_2$ such that for all $(s, u, t)\in \Delta^3_{S, T}$,
  \begin{enumerate}[$($a$)$]
    \item $|\E^u\delta A_{s, u, t}| \le \omega_1(s, t)(t-s)^{\varepsilon_1}$ almost surely;
    \item $\E^u[|\delta A_{s, u, t}|^2]^{\frac{1}{2}} \le \omega_2(s, t)^{\frac{1}{2}}(t-s)^{\varepsilon_2}$ almost surely;
    \item for any integer $n\ge 1$, there exists a constant $\Gamma_n$ independent of $s,t$ such that $\| A_{s, t} \|_{L_n(\Omega)} \le \Gamma_n(t-s)^{\varepsilon_0}$.
  \end{enumerate}
  Then for any $m \in [2, \infty)$ there exist $\ell=\ell(m, \eps_0) \in \N$ and $C = C(m, \eps_0, \eps_1, \eps_2)$ such  that for every $ (s, t) \in \Delta^2_{S, T}$,
  \begin{align*}
    \| \A_t -\A_s \|_{L_m(\Omega)}& \le C \Bigl (\Gamma_\ell (t-s)^{\eps_0} + \|\omega_1 (s, t)\|_{L_m(\Omega)}(t-s)^{\eps_1} + \|\omega_2 (s, t)^{\frac12}\|_{L_m(\Omega)}(t-s)^{\eps_2}\Bigr ).
  \end{align*}
\end{proposition}

\subsection{Main integral bound}\label{s:33}

This section is crucial for the overall proof of our main results. First, after a thorough preparation, we obtain a certain bound on additive functionals of the solutions to the SDEs~\eqref{eq:main_eq} and~\eqref{eq:main_eq_levy} in Lemma~\ref{lem:step02a}. Then, we use this bound to obtain a priori bounds on the $1$-variation of solutions to these SDEs. Since the driving noises of these SDEs share many similar features, to avoid unnecessary repetitions, we will work in a unified setting, which, as we demonstrate later, applies both to $W^H$ and $L$.

Without loss of generality, assume that the time interval is $[0,1]$. Until the end of the section, we 
fix the filtration $(\F_t)_{t\in[0,1]}$, $H\in(0,1)$ and a random process $Y : [0,1] \times \Omega \to \R^d$ adapted to $(\F_t)$. Denote 
\[
Y_{s,t}:=Y_t-\E^s Y_t,\quad (s,t)\in \Delta^2_{0,1}.
\]

\begin{assumption}
 \label{main_assumption} Suppose the following conditions hold for the process $Y$:
\begin{enumerate}[$($i$)$]
	\item  for any $(s, t) \in \Delta^2_{0, 1}$, the random variable $Y_{s, t}$ is independent of $\F_s$;
	\item for any $(s, u,t) \in \Delta^3_{0, 1}$, the random variable $\E^u Y_{s, t}$ has a density $p_{s,u,t}$ with respect to the Lebesgue measure. Furthermore, there exists a constant $C = C(H,d)>0$ such that for any $(s, u,t) \in \Delta^3_{0, 1}$
	\[
	\|p_{s,u,t}\|_{L_\infty(\R^d)}\le C \bigl( (u-s)^{-Hd}+((t-s)^{2H}-(t-u)^{2H})^{-\frac{d}2}\bigr);
	\]
	\item for any $\beta<0$ and $\rho\in[1,\infty]$ there exists a constant $C = C(H,d,\beta, \rho)$ such that for any measurable function $f\colon \R^d \to \R^d$ and any $(s, t) \in \Delta^2_{0, 1}$,
\begin{equation}
	\label{eq:hyp_smoothing_nnew}
	 \big \|\E  f(Y_{s, t}+\cdot) \big \|_{L_\rho(\R^d)} \le C \| f \|_{\B_\rho^{\beta}}(t-s)^{\beta H}.
\end{equation}
\end{enumerate}  
\end{assumption}

Throughout this section, we assume that the process $Y$ and the filtration $(\F_t)_{t\in[0,1]}$ satisfy Assumption~\ref{main_assumption}.

To simplify the notation, we introduce the following convention. We put for $\rho\in[1,\infty]$
\begin{equation}\label{Dspace}
\DD^\beta_\rho:=\B^\beta_\rho(\R^d),\,\,\text{if $\beta<0$;}\qquad \DD^0_\rho:=L_\rho(\R^d).
\end{equation}	
This allows us to avoid additional complications with the Besov spaces of zero regularity. We note in particular that by Minkowski's inequality $\|\E  f(Y_{s, t}+\cdot) \big \|_{L_\rho(\R^d)} \le \| f \|_{L_\rho(\R^d)}$, and thus~\eqref{eq:hyp_smoothing_nnew} implies
\begin{equation}
	\label{eq:hyp_smoothing_new}
	\big \|\E  f(Y_{s, t}+\cdot) \big \|_{L_\rho(\R^d)} \le C \| f \|_{\DD_\rho^{\beta}}(t-s)^{\beta H},
\end{equation}
for any $\beta\le0$ and $\rho\in[1,\infty]$. 

We begin with the following simple technical bounds. 
\begin{lemma}
For any $\beta\le 0$ and $\rho\in[1,\infty)$ there exists a constant $C= C(H,d,\beta, \rho)$ such that for any $(s, u,t) \in \Delta^3_{0, 1}$ and any measurable function $f\colon\R^d\to\R^d$ we have
\begin{equation}
	\label{eq:hyp_smoothing}
\|\E^u f(Y_{s,t})\|_{L_\rho(\Omega)} \le C \| f \|_{\DD_\rho^{\beta}}(t-u)^{\beta H}\bigl((u-s)^{-\frac{Hd}\rho}+\1_{H<\frac12}(u-s)^{-\frac{d}{2\rho}}(t-u)^{\frac{d}{2\rho}-\frac{Hd}\rho}\bigr).
\end{equation}
If, in addition, $f$ is continuous, then for any $\lambda\in(0,1]$ there exists $C= C(H,d,\lambda, \rho)$ such that for any $\delta\in \R^{d}$,
\begin{equation}\label{eq:hyp_smoothing2}
\E | f(Y_{s,t}+\delta)-f(Y_{s,t})| \le C |\delta|^\lambda\|f\|_{L_\rho(\R^d)} (t-s)^{-\lambda H-\frac{Hd}\rho}.
\end{equation}
\end{lemma}	
\begin{proof}
Let  $(s, u,t) \in \Delta^3_{0, 1}$. It follows from the definition that $Y_{s,t}=Y_{u,t}+\E^u Y_{s,t}$ and that $Y_{u,t}$ is independent of $\F_u$. Therefore, using Assumption~\ref{main_assumption}, the freezing lemma (Lemma \ref{lem:freezing}) and~\eqref{eq:hyp_smoothing_new},  we deduce for any $\beta\le0$
\begin{align}\label{impboundc}
\|\E^u f(Y_{s,t})\|_{L_\rho(\Omega)}^\rho&=\int_{\R^d} \E |f(Y_{u,t}+x)|^\rho p_{s,u,t}(x)\,dx\nn\\
&\le C\bigl((u-s)^{-Hd}+((t-s)^{2H}-(t-u)^{2H})^{-\frac{d}2}\bigr)  \|\E  f(Y_{u, t}+\cdot) \|_{L_\rho(\R^d)}^\rho\nn\\
&\le C  \| f \|_{\DD_\rho^{\beta}}^\rho (t-u)^{\beta \rho H}\bigl((u-s)^{-Hd}+((t-s)^{2H}-(t-u)^{2H})^{-\frac{d}2}\bigr).
\end{align}
If $H\ge1/2$, then $(t-s)^{2H}-(t-u)^{2H}\ge (u-s)^{2H}$. Substituting this into the above inequality, we deduce the desired bound~\eqref{eq:hyp_smoothing}.

If 	$H<1/2$, then  $(t-s)^{2H}-(t-u)^{2H}\ge 2H (u-s)(t-s)^{2H-1}$. Therefore,
\[
((t-s)^{2H}-(t-u)^{2H})^{-\frac{d}2}\le C (u-s)^{-\frac{d}2}(t-s)^{\frac{d}2-Hd}\le C 
(u-s)^{-Hd}+C(u-s)^{-\frac{d}2}(t-u)^{\frac{d}2-Hd}
\]
for $C=C(H,d)$. 
Combining this with~\eqref{impboundc}, we obtain~\eqref{eq:hyp_smoothing}.

Next, to prove~\eqref{eq:hyp_smoothing2}, we fix $\lambda\in(0,1]$. We note that for any continuous $g\colon\R^d\to\R^d$, we have from~\eqref{eq:hyp_smoothing_new}
\begin{equation*}
|\E g(Y_{s,t})|\le \|\E  g(Y_{s,t}+\cdot)\|_{L_\infty(\R^d)}\le C \|g\|_{\DD^{-\lambda-d/\rho}_\infty} (t-s)^{-\lambda H-\frac{Hd}\rho}\le C \|g\|_{\DD^{-\lambda}_\rho} (t-s)^{-\lambda H-\frac{Hd}\rho},
\end{equation*}
where $C=C(H,d,\lambda,\rho)$ and in the first inequality, we used that the function $x\mapsto \E  g(Y_{s,t}+x)$ is continuous. Therefore, using~\eqref{besovpr}, we deduce for any $\delta\in\R^d$
\begin{equation*}
|\E  f(Y_{s,t}+\delta)-f(Y_{s,t})|\le C \|f(\cdot)-f(\delta+\cdot)\|_{\DD^{-\lambda}_\rho} (t-s)^{-\lambda H-\frac{Hd}\rho}
\le C |\delta|^\lambda\|f\|_{\B^{0}_\rho} (t-s)^{-\lambda H-\frac{Hd}\rho},
\end{equation*}
for $C=C(H,d,\lambda,\rho)$. Since $\|\cdot\|_{\B^{0}_\rho}\le C \|\cdot\|_{L_\rho(\R^d)}$, this implies~\eqref{eq:hyp_smoothing2}.
\end{proof}

Now we are ready to prove a preliminary bound on an additive functional of $Y$. Recall the convention~\eqref{Dspace}.
\begin{lemma}
  \label{lem:step01}
  Let  $f: [0, 1] \times \R^d \to \R^d$ be a bounded  function and suppose that $f\in L_q([0, 1], \DD_\rho^{\beta}(\R^d))$, where 
  \begin{equation}
    \label{eq:fsthyplemma} q\in[1,\infty), \quad \rho \in [2,\infty),  \quad \frac{1}{q} + \frac{Hd}{\rho} - H \beta <1, \quad  \text{and}\quad \beta\in(-\frac{1}{2H},0].
  \end{equation}
 Then there exists a constant  $C = C(H, d,\beta, \rho, q)$ such that for any  $(s,t)\in\Delta^2_{0,1}$ one has
  \begin{equation}\label{bstep01}
    \Big \| \int_s^t f_r(Y_{s, r}) \diff r \Big \|_{L_\rho (\Omega)} \le C\| f \|_{L_q([s, t], \DD_\rho^{\beta}(\R^d))} (t-s)^{1-\frac{1}{q}-\frac{Hd}{\rho} +\beta H}.
  \end{equation}
\end{lemma}
\begin{proof}
Fix $ 0 \le S_{0} < S \le T \le 1$. We will show that~\eqref{bstep01} holds for an integral between $S_{0}$ and $T$, and to do that, we will apply the stochastic sewing lemma (Proposition~\ref{prop:simple_sewing}) to the random variables defined by
\begin{equation*}
    A_{s, t} := \E^s\int_s^t f_r(Y_{S_0, r}) \diff r \quad \quad \text{and} \quad \quad \A_t:=\int_{S_0}^t f_r(Y_{S_0, r})\diff r
\end{equation*}
for $(s,t)\in\Delta^2_{S,T}$.

First, we verify the convergence hypothesis (\ref{eq:convergence}).
  Let $\Pi_N:=\{s=t^N_0,\dots,t^N_{k(N)}=t\}$ be a sequence of partitions of $[s,t] \subseteq [S, T]$ with $\lim_{N\to\infty}\Di{\Pi_N}\to0$.
  Clearly, for any $(u, v)\in \Delta^2_{S,T}$, $A_{u, v} = \E^u[\A_v-\A_u]$ so $R^N_{i+1} := \A_{t_{i+1}^N} - \A_{t_i^N} - A_{t_i^N, t_{i+1}^N}$ is a $(\F_{t_{i+1}^N})_i$-martingale difference sequence. Therefore, using the Burkholder--Davis--Gundy inequality (which is an equality in this case), and bounding $R^N_{i+1} $ using the $L^{\infty}$ norm of $f$, we obtain 
  \begin{align*}
   \Big \| \sum_{i=0}^{k(N)-1}
   A_{t^N_i,t_{i+1}^N} - \A_T + \A_S\Big \|_{L_2(\Omega)}^2 &= \Bigl\| \sum_{i=0}^{k(N)-1}R_{i+1}^N \Bigr\|_{L_2(\Omega)}^2\\
  & = \sum_{i=0}^{k(N)-1}\|R_{i+1}^N \|_{L_2(\Omega)}^2\\
  & \le C \|f\|_{L_\infty([0,T]\times\R^d)}^2\sum_{i=0}^{k(N)-1}(t_{i+1}^N-t_i^N)^2\\
  & \le C \|f\|_{L_\infty([0,T]\times\R^d)}^2 (T-S) \Di{\Pi_N},
  \end{align*}
  which goes to $0$ as $|\Pi_N|\to0$. Hence, the condition~\eqref{eq:convergence} holds.

  Next, we verify the assumption (a) of Proposition~\ref{prop:simple_sewing}.
  For $(s,u,t)\in\Delta^3_{S,T}$, we have 
  \begin{equation*}
  \delta A_{s, u, t}= A_{s,t}-A_{s, u} - A_{u, t} = \E^s[\A_t-\A_u] - \E^u[\A_t-\A_u].
  \end{equation*}
  Therefore, $\E^s\delta A_{s, u, t}=0$ and condition (a) holds.

  Finally, let us check the assumption (b) of Proposition~\ref{prop:simple_sewing}. 
  We recall the inequality~\eqref{eq:hyp_smoothing} and we derive 
    \begin{align*}
    &\| A_{s, t} \|_{L_\rho(\Omega)}\\
     &\quad\le\int_s^t \|\E^s f_r(Y_{S_0,r})\|_{L_\rho(\Omega)}\diff r\\
     &\quad\le C \int_s^t \| f_r \|_{\DD_\rho^{\beta}} \Bigl ((s-S_0)^{-\frac{Hd}{\rho}}(r-s)^{\beta H}+\1_{H<\frac{1}{2}}(s-S_0)^{-\frac{d}{2 \rho}}(r-s)^{\frac{d}{2 \rho}-\frac{Hd}{\rho} +\beta H} \Bigr )\diff r\\
    &\quad\le C\| f \|_{L_q([s, t], \DD_\rho^{\beta}(\R^d))} (S-S_0)^{-\frac{Hd}{\rho}}(t-s)^{1-\frac{1}{q}+ \beta H}\\
    &\qquad+\1_{H<\frac{1}{2}}C\| f \|_{L_q([s, t], \DD_\rho^{\beta}(\R^d))}(S-S_0)^{-\frac{d}{2 \rho}}(t-s)^{1-\frac{1}{q} +\frac{d}{2 \rho}-\frac{Hd}{\rho} + \beta H},
    \end{align*}
where $C=C(H,d,\beta,\rho)$ and in the last inequality we use the H\"older inequality, which is applicable since $\beta H>-1+1/q$ by~\eqref{eq:fsthyplemma}. 
  The function $(s, t) \mapsto \| f \|_{L_q([s, t], \DD_\rho^{\beta}(\R^d))}^q$ is a control and so is $(s, t) \mapsto t-s$. Therefore, by \cite[Exercise~1.10]{FV10} the function
  \[
  w(s, t):= \| f \|_{L_q([s, t], \DD_\rho^{\beta}(\R^d))}^{\frac{1}{1+\beta H}} (t-s)^{\frac{1-\frac1q+\beta H}{1+\beta H}}
  \]
  is a control. Hence, we get  
  \[
  \| A_{s, t} \|_{L_\rho(\Omega)}\le  C w(s,t)^{1+\beta H} \bigl((S-S_0)^{-\frac{Hd}{\rho}}+\1_{H<\frac{1}{2}}(S-S_0)^{-\frac{d}{2 \rho}}(T-S)^{\frac{d}{2 \rho}-\frac{Hd}{\rho}}\bigr).
  \]
  By~\eqref{eq:fsthyplemma}, $ 1 +\beta H > \frac{1}{2}$. Therefore, the assumption (b) of Proposition~\ref{prop:simple_sewing} holds. 
  
  We also see that $\rho\ge2$ by assumption. Thus, all the conditions of Proposition~\ref{prop:simple_sewing} are satisfied, and we obtain for some $C=C(H,d,\beta,\rho)$ that
  \begin{align*}
    \| \A_T - \A_{S} \|_{L_\rho(\Omega)}& \le C\| f \|_{L_q([S_0, T], \DD^{\beta}_\rho (\R^d))} (S-S_0)^{-\frac{Hd}{\rho}}(T-S)^{1-\frac{1}{q} + \beta H}\\
    &\quad+\1_{H<\frac{1}{2}}C\| f \|_{L_q([S_0, T], \DD^{\beta}_\rho (\R^d))}( S-S_0)^{-\frac{d}{2 \rho}}(T-S)^{1-\frac{1}{q} +\frac{d}{2 \rho}-\frac{Hd}{\rho}+\beta H}.
  \end{align*}
	
Recall that $(S,T)$ were arbitrary elements of $\Delta^2_{S_0,1}$, $S\neq S_0$.  Thanks to~\eqref{eq:fsthyplemma}, we have $1-\frac{1}{q} - \frac{Hd}{\rho} + \beta H > 0$. Therefore, we can apply the taming singularities lemma (Proposition~\ref{prop:taming_singularities}) to obtain that for all $S_0 \le T\le 1$,
  \begin{equation*}
    \Big \| \int_{S_0}^T f_r(Y_{S_0, r}) \diff r \Big \|_{L_\rho (\Omega)} = \| \A_T-\A_{S_0} \|_{L_\rho (\Omega)} \le C\| f \|_{L_q([S_0, T], \DD^{\beta}_\rho (\R^d))} (T-S_0)^{1-\frac{1}{q}-\frac{Hd}{\rho}+ \beta H},
  \end{equation*}
  which is the desired bound~\eqref{bstep01}.
\end{proof}

Using the above result, it is easy to bound the difference of two additive functionals. 

\begin{corollary}
	\label{cor:tr_mmnts}
Let $\lambda\in[0,1]$.   Let  $f: [0, 1] \times \R^d \to \R^d$ be a bounded  function and suppose that $f\in L_q([0, 1], \DD_\rho^{\beta}(\R^d))$, where $\beta\le0$ and 
	\begin{equation}
	\label{hyp:tr_mmnts} q\in[1,\infty), \quad \rho \in [2,\infty),  \quad \frac{1}{q} + \frac{Hd}{\rho}- H(\beta-\lambda) < 1,  \quad\text{and}\quad\beta-\lambda\in(-\frac{1}{2H},0).
\end{equation}
Then, there exists a constant  $C = C(H, d,\beta, \rho, q,\lambda)$ such that for any  $(s,t)\in\Delta^2_{0,1}$, $\delta\in\R^d$,  
and any deterministic measurable function $x\colon[0,1]\to\R^d$ one has	
\begin{equation*}
	\Big \| \int_s^t f_r(Y_{s, r} + x_r+\delta)-f_r(Y_{s, r} + x_r)\diff r \Big \|_{L_2(\Omega)} \le C |\delta |^{\lambda}\| f \|_{L_q([s, t], \DD_{\rho}^{\beta}(\R^d))} (t-s)^{1-\frac{1}{q}-\frac{Hd}{\rho} +\beta H- \lambda H}.
\end{equation*}
\end{corollary}
\begin{proof}
Define $h_r:= f_r(\cdot + x_{r} + \delta) - f_r(\cdot + x_{r})$.
	We apply Lemma~\ref{lem:step01} to the function $h$ instead of $f$ and $\beta - \lambda$ instead of $\beta$. We see that the condition~\eqref{eq:fsthyplemma} holds thanks to~\eqref{hyp:tr_mmnts}.
	We obtain that there exists $C = C(H, d, \beta, \rho, q,\lambda)$ such that for every $(s,t)\in\Delta_{0,1}^2$,
	\begin{align*}
		&\Big \| \int_s^t f_r(Y_{s, r} + x_r+\delta) - f_r(Y_{s, r}+x_r) \diff r \Big \|_{L_\rho(\Omega)}\\
		&\qquad= \Big \| \int_s^t h_r(Y_{s, r}) \diff r \Big \|_{L_\rho(\Omega)}\\
		&\qquad\le C \| h \|_{L_q([s, t], \DD_{\rho}^{\beta -\lambda}(\R^d))} (t-s)^{1-\frac{1}{q}-\frac{Hd}{\rho} + \beta H- \lambda H}.
		\\&\qquad \le C|\delta |^{\lambda}\| f \|_{L_q([s, t], \DD_{\rho}^{\beta}(\R^d))} (t-s)^{1-\frac{1}{q}-\frac{Hd}{\rho} + \beta H- \lambda H},
	\end{align*} 
 where the last inequality follows from~\eqref{besovpr} and from $\|\cdot\|_{\B^{0}_\rho}\le C \|\cdot\|_{L_\rho(\R^d)}$ in the case $\beta = 0$.
  We stress that the constant does not depend on the deterministic process $x$. This implies the desired bound since $\rho\ge2$ by assumption.
\end{proof}

Our next step is to establish the same bound as in Lemma~\ref{lem:step01} but for arbitrarily large moments of the integral. This is achieved by using the John--Nirenberg inequality \cite[Theorem 2.3]{Le22}, \cite[Theorem 1.3]{Le22b}, see also  \cite[Proposition~3.2]{BDG24}.

\begin{lemma}
  \label{lem:JN_mom}
   Let  $m\ge1$, $f: [0, 1] \times \R^d \to \R^d$ be a bounded function, and suppose that $f\in L_q([0, 1], \DD_\rho^{\beta}(\R^d))$, where the parameters $q, \rho, H, d, \beta$ satisfy~\eqref{eq:fsthyplemma}. Then there exists a constant $C = C(H,d,\beta,\rho, q, m)$ such that for every adapted measurable process $\psi\colon [0,1]\times\Omega\to\R^d$ and any $(s,t)\in\Delta^2_{0,1}$,
  \begin{equation}\label{JN}
    \Big \| \int_s^t f_r(Y_{ r} + \psi_s)\diff r \Big \|_{L_m(\Omega)} \le C \| f \|_{L_q([s, t], \DD_\rho^{\beta}(\R^d))} (t-s)^{1-\frac{1}{q}-\frac{Hd}{\rho} + \beta H}.
  \end{equation}
\end{lemma}

\begin{proof}
Fix an adapted measurable process $\psi\colon [0,1]\times\Omega\to\R^d$, $(S,T)\in\Delta^2_{0,1}$, and define
\[
\A_t := \int_S^t f_r(Y_{ r} + \psi_S) \diff r,\quad t\in[S,T].
\]
Recall that for $0\le s\le r$ we have $Y_{r}=Y_{s,r}+\E^s Y_{r}$ and $Y_{s,r}$ is independent of $\F_s$ by Assumption~\ref{main_assumption}. Therefore, for any $(s,t)\in\Delta_{S,T}^2$ we can compute using Lemma~\ref{lem:freezing} the conditional expectation $\E^{s}|\A_{t}-\A_{s}|$,
\[
    \E^s|\A_t-\A_s| = \E^s\Big |\int_s^t f_r(Y_{r}+ \psi_S) \diff r \Big | = \E \Big| \int_s^t f_r(x_r + Y_{s, r}) \diff r \Big |\bigg\rvert_{x_r:=\E^s Y_{ r} + \psi_S}.
  \]
When computing the expectation in the last expression, the variable $x$ is considered deterministic and the resulting expectation is evaluated at $x_r:=\E^s Y_{ r} + \psi_S$. For such deterministic $x_{r}$, we apply the Lemma~\ref{lem:step01} to the function $h_{r} := f_{r}(x_{r} + \cdot)$ instead of $f$. Now using the fact that $\| \cdot  \|_{L_1(\Omega)} \le \| \cdot  \|_{L_\rho (\Omega)}$ and that $\|f\|_{L_q([S, T], \DD_\rho^{\beta}(\R^d))} = \|h\|_{L_q([S, T], \DD_\rho^{\beta}(\R^d))}$, we obtain that there exists $C=C(H, d, \beta,\rho, q)$ such that for all $(s,t)\in\Delta_{S,T}^2$ one has
  \[
    \E^s|\A_t-\A_s|  \le  C\| f \|_{L_q([S, T], \DD_\rho^{\beta}(\R^d))} (T-S)^{1-\frac{1}{q}-\frac{Hd}{\rho}+\beta H}.
  \]
We see that the process $\A$ is continuous because $f$ is bounded. Therefore, by John--Nirenberg's theorem (Proposition~\ref{prop:John-Nierenberg}), we get that for all $m \ge 1$ there exists $C = C(H, d, \beta, \rho, q, m)$ such that
  \[
    \Big \| \int_S^T f_r(Y_{ r} + \psi_S)\diff r \Big \|_{L_m(\Omega)} =\|\A_T-\A_S\|_{L_m(\Omega)}\le C \| f \|_{L_q([S, T], \DD_\rho^{\beta}(\R^d))} (T-S)^{1-\frac{1}{q}-\frac{Hd}{\rho}+ \beta H},
  \]
  which is~\eqref{JN}.
\end{proof}
Now, we are ready to establish a bound on all the moments of additive functionals of the noise $Y$ perturbed by a drift of bounded variation. We will employ the Rosenthal--Burkholder stochastic sewing lemma \cite[Theorem 3.6]{BLM23}. To verify the assumptions of the lemma, we will use Lemma~\ref{lem:step01}, Corollary~\ref{cor:tr_mmnts}, and Lemma~\ref{lem:JN_mom}. First, we assume additionally that $f$ is bounded and Lipschitz, and later we will remove these conditions. Compared with \cite[Lemma 4.6]{BLM23}, we need to handle the terms $|\E^u\delta A_{s,u,t}|$ and $\E^u|\delta A_{s,u,t}|^2$ more delicately here in order to obtain the optimal condition and break the subcritical barrier.

\begin{lemma}
  \label{lem:step02a}
Let $m \in [2, \infty)$. Let  $f: [0, 1] \times \R^d \to \R^d$ be a bounded function, and suppose that $f\in L_q([0, 1], L_p(\R^d))$, where $p, q \in [1, \infty)$ satisfy the main condition (\ref{eq:main_cond}). Assume further that $f\in L_\infty([0,1],\C^1(\R^d))$. 
  Then, there exists a constant $C=C(H,d, p, q,  m)$ such that for any adapted measurable process $\psi : [0, 1] \times\Omega\to \R^d$ with
\begin{equation}
    \label{fin1var} \big \| [\psi ]_{\C^{1-\var}([0, 1])} \big \|_{L_m(\Omega)} <  \infty,
  \end{equation}
and for any $(s,t)\in\Delta^2_{0,1}$, one has
  \begin{align}
    \label{3bound}& \Big \| \int_s^t f_r(Y_{r} + \psi_r)\diff r \Big \|_{L_m(\Omega)}\nn\\
     &\quad\le C \| f \|_{L_q([s, t], L_p(\R^d))}(t-s)^{1-\frac{1}{q}-\frac{Hd}{p}} \Bigl (1 + \big\|[\psi ]_{\C^{1-\var}([s, t])}\big \|_{L_m(\Omega)}^{1-\frac{1}{q}}(t-s)^{-H(1-\frac{1}{q})}\Bigr ).
  \end{align}
\end{lemma}

\begin{proof}
  For any $(s,t)\in\Delta^2_{0,1}$ and for any adapted measurable process $\psi : [0, 1] \times\Omega\to \R^d$ satisfying~\eqref{fin1var}, we define 
 \[
    A_{s, t} := \int_s^t f_r(Y_{r} + \psi_s) \diff r \quad \text{and} \quad \A_t = \int_0^t f_r(Y_{r} + \psi_r)\diff r,
  \]
 and we want to use the stochastic sewing lemma from Proposition~\ref{prop:rosenthal_sewing}.

  \textbf{Step 1.}
  We prove the convergence (\ref{eq:convergence}).
  Let $\Pi_N := \{S = s_0^N < \dots < s_{k(N)}^N = T\}$ be a sequence of partitions of $[S, T] \subseteq [0, 1]$ such that $\Di{\Pi_N} \to 0$ when $N \to \infty $.
Using the triangle inequality and the fact that $f$ is Lipschitz in space, we get
 \begin{align*}
\Bigl|\A_T-\A_S- \sum_{i=0}^{k(N)-1} A_{s^N_i,s_{i+1}^N}\Bigr| &\le \sum_{i=0}^{k(N)-1} \int_{s_i^N}^{s_{i+1}^N} \bigl |f_r(Y_{ r}+\psi_r)-f_r(Y_{ r}+\psi_{s_i^N})\bigr |\diff r.\\
&\le \|f\|_{L_\infty([0,1],\C^1(\R^d))} \sum_{i=0}^{k(N)-1} \int_{s_i^N}^{s_{i+1}^N} [\psi ]_{\C^{1-\var}([s_i^N, s_{i+1}^N])}\diff r\\
&    \le   \|f\|_{L_\infty([0,1],\C^1(\R^d))}[\psi ]_{\C^{1-\var}([S, T])} \Di{\Pi_N}.
\end{align*}
Since by~\eqref{fin1var}, $[\psi ]_{\C^{1-\var}([S, T])}<\infty$ almost surely,  we see that the left-hand side of the above inequality goes to $0$ almost surely (and thus in probability) as $N\to\infty$. This proves condition~\eqref{eq:convergence}.

  \textbf{Step 2.}
We verify the condition (a) of Proposition~\ref{prop:rosenthal_sewing}. Clearly, for $(s,u,t)\in\Delta^3_{0,1}$, we have
  \[
    \delta A_{s, u, t} = \int_u^t f_r(Y_{r}+\psi_s) - f_r(Y_{r}+\psi_u) \diff r.
  \]
 Recall that $Y_r=\E^u Y_r+ Y_{u,r}$. Applying the triangle inequality we obtain
\begin{equation*}
    |\E^u[\delta A_{s, u, t}]| 
    \le \int_u^t \big |\E^{u} \big [f_r(Y_{u, r}+(\E^u Y_{ r} + \psi_u)+(\psi_s-\psi_u)) -f_r(Y_{u, r} + (\E^u Y_{ r} + \psi_u))\big ]\big |\diff r
\end{equation*}
We see that for any $r \in [u, t]$ the random variables  $\E^u Y_{ r} + \psi_u$ and $\psi_s-\psi_u$ are $\F_u$-measurable, while $Y_{u,r}$ is independent of $\F_u$ by Assumption~\ref{main_assumption}. Hence, by applying Lemma~\ref{lem:freezing}, the conditional expectation with respect to $\F_{u}$ can be expressed by first ``freezing'' these random variables  and then taking the expectation, and we get
  \begin{align}\label{Ybound1}
    |\E^u[\delta A_{s, u, t}]| 
    &\le \int_u^t \big |\E \big [f_r(Y_{u, r}+x_r+\delta) -f_r(Y_{u, r} + x_r)\big ]\big |\diff r\bigg\rvert_{\substack{x_r:=\E^u Y_{ r} + \psi_u\\\delta:=\psi_s-\psi_u}} .
  \end{align}
Let 
\begin{equation}\label{lambdafirst}
\lambda\in(0,1]
\end{equation}
By~\eqref{eq:hyp_smoothing2}, we have 
\begin{equation*}
\big|\E \big [f_r(Y_{u, r}+x_r+\delta) -f_r(Y_{u, r} + x_r)\big ]\big |\le  C |\delta|^\lambda \|f_r\|_{L_p(\R^d)}(r-u)^{-\frac{Hd}{p}-\lambda H}
\end{equation*}
for $C=C(H,d,\lambda,p)$ not depending on $\psi$. Substituting this back into~\eqref{Ybound1} and applying the H\"older inequality, we get 
  \begin{align}\label{checking1}
    |\E^u[\delta A_{s, u, t}]|& \le C |\psi_s-\psi_u|^{\lambda} \int_u^t \| f_r \|_{L_p(\R^d)} (r-u)^{-\frac{Hd}{p}-\lambda H}\diff r
    \nn\\& \le C [\psi ]_{\C^{1-\var}([s, t])}^{\lambda}  \| f \|_{L_q([u, t], L_p(\R^d))}(t-u)^{1-\lambda -\frac{1}{q}}(t-u)^{\lambda -\frac{Hd}{p}-\lambda H},
  \end{align}
provided that 
\begin{equation}\label{lambdasecond}
1-\frac{1}{q}-\frac{Hd}{p}-\lambda H>0.
\end{equation}

Since $\| f \|^q_{L_q([s, t], L_p(\R^d))}$, $[\psi ]_{\C^{1-\var}([s, t])}$ and $(t-s)$ are three controls, if $1- \lambda -\frac{1}{q} \ge 0$, by \cite[Exercise~1.10]{FV10}, we have that
\begin{equation}\label{claim1}
[\psi ]_{\C^{1-\var}([s, t])}^{\lambda}  \| f \|_{L_q([u, t], L_p(\R^d))}(t-u)^{1-\lambda -\frac{1}{q}}\, \text{is a random control.} 
\end{equation}

To verify  condition (a) of Proposition~\ref{prop:rosenthal_sewing} it remains to check that this power $ \lambda -\frac{Hd}{p}-\lambda H$ is larger than $0$.
We combine the two restrictions on $\lambda$, \eqref{lambdafirst} and \eqref{lambdasecond}, and choose therefore\footnote{Actually, the optimal choice here would be  $\lambda:=(\frac{1-\frac{1}{q}-\frac{Hd}{p}-\eps}H)\wedge1$ for some small $\eps>0$. However, this choice would lead to exactly the same condition \eqref{eq:main_cond}.}
\begin{equation}\label{lambdaopt}
\lambda=1-\frac1q.
\end{equation}

Using \eqref{eq:main_cond}, we see that with such choice of $\lambda$ we have
\begin{equation*}
\lambda -\frac{Hd}{p}-\lambda H>0.
\end{equation*}	
Recalling now \eqref{claim1}, we see that condition (a) of Proposition~\ref{prop:rosenthal_sewing} is satisfied.

  \textbf{Step 3.}
  We verify condition (b) of Proposition~\ref{prop:rosenthal_sewing}. Let $(s,u,t)\in\Delta^3_{0,1}$. Arguing as above with Lemma~\ref{lem:freezing}, using Assumption~\ref{main_assumption} we get
  \begin{align*}
    \E^u[|\delta A_{s, u, t}|^2] =\E\Bigl[ \int_u^t f_r(Y_{u, r}+x_r+\delta) -f_r(Y_{u, r} + x_r)\diff r\Bigr]^2 \bigg\rvert_{\substack{x_r:=\E^u Y_{ r} + \psi_u\\\delta:=\psi_s-\psi_u}}.
  \end{align*}
  Define
\begin{equation*}
    \rho := 2\vee p, \quad \quad \beta := \frac{d}{\rho}-\frac{d}{p}\le 0, \quad \quad \lambda := \frac{1}{2}\Bigl(1 -\frac{1}{q}\Bigr).
  \end{equation*}
  We verify the hypothesis of Corollary~\ref{cor:tr_mmnts} with these values of $\rho$, $\beta $ and $\lambda $.
  By the Besov embeddings, $f\in L_q([0, 1], \DD_{\rho}^{\beta}(\R^d))$. We also have $\rho \ge 2$. Furthermore, using~\eqref{eq:main_cond} we get
  \[
    \frac{1}{q} + \frac{Hd}{\rho}+ H(2\lambda - \beta) = \frac{1-H}{q} + H + \frac{Hd}{p}< 1.
  \]
Since $\lambda \le 2 \lambda $, the third part of condition (\ref{hyp:tr_mmnts}) is satisfied. 
Finally, we note that by definition $\beta = \frac{d}{p}(\frac{p}{\rho} - 1)$ and $\rho \le 2p$ so $\beta \ge -\frac{d}{2p}$.
  Using this fact and (\ref{eq:main_cond}) we get
  \[
    2H(\lambda - \beta)=H-\frac{H}q-2H\beta \le H - \frac{H}{q} + \frac{Hd}{p} < 1 -\frac{1}{q} < 1
  \]
  so $\beta-\lambda>-1/(2H)$ and the last part of~\eqref{hyp:tr_mmnts} is satisfied.
  
  Therefore, applying Corollary~\ref{cor:tr_mmnts} and using the Besov embedding $L_p(\R^d) \subseteq \DD_{\rho}^{\beta}(\R^d)$, we get that there exists $C = C(H, d, p, q)$ not depending on $\psi$ such that
 \begin{align*}
 \E^u [|\delta A_{s, u, t}|^2]^{\frac{1}{2}} &\le C \| f \|_{L_q([s, t], L_p(\R^d))}|\psi_s - \psi_u|^{\lambda} (t-s)^{1-\frac{1}{q}-\frac{Hd}{p}-\lambda H}\\
& \le C \| f \|_{L_q([s, t], L_p(\R^d))} [\psi ]_{\C^{1-\var}([s, t])}^{\lambda}  (t-s)^{1-\frac{1}{q}-\frac{Hd}{p}-\lambda H}.
 \end{align*}
The power $1 - \frac{1}{q} - \frac{Hd}{p} - \lambda H$ is positive by the third part of~\eqref{hyp:tr_mmnts} that we just proved. As before, the term $\| f \|_{L_q([s, t], L_p(\R^d))} [\psi ]_{\C^{1-\var}([s, t])}^{\lambda}$ is a random control to the power $\frac{1}{q} + \lambda = \frac{1}{2} + \frac{1}{2q} \ge \frac{1}{2}$. Therefore, it is also a control to the power $\frac{1}{2}$, and the condition (b) of Proposition~\ref{prop:rosenthal_sewing} is satisfied.

  \textbf{Step 4.}
  We verify the condition (c) of Proposition~\ref{prop:rosenthal_sewing}.
  We take 
  \[
    \rho := 2\vee p, \quad \quad  \beta := \frac{d}{\rho}-\frac{d}{p},
  \]
  and we verify all the conditions of Lemma~\ref{lem:JN_mom}. Clearly, $\rho \ge 2$.
 Assumption~\eqref{eq:main_cond} yields
  \[
    \frac{1}{q}+\frac{Hd}{\rho} -\beta H= \frac{1}{q} + \frac{Hd}{p} < 1 -H + \frac{H}{q} < 1,
  \]
  so the third part of (\ref{eq:fsthyplemma}) is satisfied.
  Finally,  $\beta \ge -\frac{d}{2p}$ and because $\frac{Hd}{p} < 1$ by (\ref{eq:main_cond}) the last part of~\eqref{eq:fsthyplemma} is satisfied. 
  
  Thus, all the conditions of Lemma~\ref{lem:JN_mom} hold. Therefore, we deduce that for all integers $n \ge 1$ there exists a constant $C = C(H, d, p, q, n)$ that does not depend on $\psi$ such that for all $(s,t)\in\Delta_{0,1}$,
  \[
  \| A_{s, t} \|_{L_n(\Omega)}=\Bigl\| \int_s^t f_r(Y_{r} + \psi_s) \diff r\Bigr\|_{L_n(\Omega)} \le C \| f \|_{L_q([s, t], L_p(\R^d))} (t-s)^{1-\frac{1}{q}-\frac{Hd}{p}},
  \]
  where we also used the Besov embedding $\| \cdot \|_{\DD_\rho^{\beta}(\R^d)} \le \| \cdot \|_{L_p(\R^d)} $ (recall the convention~\eqref{Dspace}). Thus, the condition (c) of Proposition~\ref{prop:rosenthal_sewing}
 hold.
 
  \textbf{Step 5.}
  Thus, all the assumptions of Proposition~\ref{prop:rosenthal_sewing} are satisfied. Therefore, there exists a constant $C = C(H, d, p, q, m)$  such that for any  $(s,t)\in\Delta_{0,1}^{2}$,
  \begin{align*}
    &\Big \| \int_s^t f_r(Y_{ r} + \psi_r)\diff r \Big \|_{L_m(\Omega)}\\
    &\,\, \le C \| f \|_{L_q([s, t], L_p(\R^d))}(t-s)^{1-\frac{1}{q}-\frac{Hd}{p}}
    \\ &\,\, \quad + C \| f \|_{L_q([s, t], L_p(\R^d))} \big\|[\psi ]^{1-\frac{1}{q}}_{\C^{1-\var}([s, t])}\big \|_{L_m(\Omega)}(t-s)^{1-H-\frac{1-H}{q}-\frac{Hd}{p}}
    \\ &\,\, \quad + C \| f \|_{L_q([s, t], L_p(\R^d))} \big\|[\psi ]^{\frac{1}{2}(1-\frac{1}{q})}_{\C^{1-\var}([s, t])}\big \|_{L_m(\Omega)}(t-s)^{1-\frac{H}{2}-\frac{2-H}{2q}-\frac{Hd}{p}}\\
&\,\, \le C\| f \|_{L_q([s, t], L_p(\R^d))}(t-s)^{1-\frac{1}{q}-\frac{Hd}{p}} \Bigl (1 + \frac{\big\|[\psi ]_{\C^{1-\var}([s, t])}\big \|^{1-\frac{1}{q}}_{L_m(\Omega)}}{(t-s)^{H(1-\frac{1}{q})}} + \frac{\big\|[\psi ]_{\C^{1-\var}([s, t])}\big \|^{\frac{1}{2}(1-\frac{1}{q})}_{L_m(\Omega)}}{(t-s)^{\frac{H}{2}(1-\frac{1}{q})}}\Bigr ).
  \end{align*}
Using inequality $1 + x + \sqrt{x} \le 2(1 + x)$ valid for all $x\ge0$, we get the desired bound~\eqref{3bound}. We stress again that the constant $C$ does not depend on $\psi$.
\end{proof}

\begin{remark}
Bound \eqref{checking1} generalizes \cite[inequality (4.16)]{BLM23} by introducing a parameter $\lambda\in[0,1]$. The latter inequality corresponds to $\lambda = 1$. It is easy to see that in the current setting, the choice $\lambda=1$ leads to a suboptimal condition (recall \eqref{lambdasecond})
\begin{equation}\label{subopt}
1-\frac{1}{q}-\frac{Hd}{p}-H>0,
\end{equation}
while optimizing over $\lambda\in[0,1]$ yields condition \eqref{eq:main_cond}, which is sharp.
\end{remark}	

Next, let us remove the restriction in Lemma~\ref{lem:step02a} that $f$ is a smooth function. 

\begin{corollary}
  \label{cor:onlymeasurable}
Let $m\ge2$. Let  $f\in L_q([0, 1], L_p(\R^d))$, where  $p, q \in [1, \infty)$ satisfy~\eqref{eq:main_cond}. 
Then there exists a constant $C=C(H,d, p, q,  m)$ such that  for any  $(s,t)\in\Delta^2_{0,1}$ and for any adapted measurable process $\psi : [0, 1] \times\Omega\to \R^d$ satisfying~\eqref{fin1var}, the bound~\eqref{3bound} holds.
\end{corollary}
The proof of this result is exactly the same as \cite[Lemma 4.6, steps 2 to 4]{BLM23}.

Finally, we are ready to provide a priori bounds to solutions of the equation 
 \begin{equation}
 	\label{SDEfunc}
 	\Gtag{x;f;Y} \psi_t = \int_0^t f_r(x + \psi_r + Y_{r})\diff r, \quad t\in [0, 1],\,\,x\in\R^d,
\end{equation}
  where the integral is understood as a Lebesgue integral. This equation is closely related to our main SDEs~\eqref{eq:main_eq} and~\eqref{eq:main_eq_levy}, by taking there $Y=W^H$ or $Y=L$ and $X=Y+\psi+x$. We say that $\psi$ solves \Gref{x;f;Y} if $\P\bigl(\text{\Gref{x;f;Y} holds for all $t\in[0,1]$}\bigr)=1$.
  
\begin{corollary}
  \label{cor:step02b}
Let $m\ge2$, $x\in\R^d$. Let  $f\in L_q([0, 1], L_p(\R^d))$, where  $p, q \in [1, \infty)$ satisfy the main condition (\ref{eq:main_cond}).
  Let $\psi : [0, 1] \to \R^d$ be a solution to \Gref{x;f;Y} adapted to the filtration $(\F_t)$ and such that $\|   [\psi ]_{\C^{1-\var}([0, 1])} \|_{L_m(\Omega)} < \infty$.

  Then there exists a constant $C=C(H,d, p, q,  m)$ such that for any $(s,t)\in\Delta^2_{0,1}$ one has
  \begin{equation}\label{c313}
    \bigl \| [\psi ]_{\C^{1-\var}([s, t])} \bigr \|_{L_m(\Omega)} \le C \bigl (\| f \|_{L_q([s, t], L_p(\R^d))} + \| f \|_{L_q([s, t], L_p(\R^d))}^q\bigr )(t-s)^{1-H-\frac{Hd}{p}-\frac{1-H}{q}}.
  \end{equation}
Furthermore, there exist $\gamma >0$ and $C = C(H, d, p, q, m)>0$ such that
  \begin{equation}
    \label{milshef} \P (\|\psi\|_{\C^\gamma([0,1])} \ge r) \le C\bigl (1+ \| f \|_{L_q([s, t], L_p(\R^d))}^q\bigr )^m r^{-m}, \quad r>0.
\end{equation}
\end{corollary}

\begin{remark}
We note that the power of the factor $(t-s)$ in the right-hand side of~\eqref{c313} is positive by condition~\eqref{eq:main_cond}. 
\end{remark}
\begin{proof}
	Fix $m\ge2$. 
The hypothesis allow us to apply Corollary~\ref{cor:onlymeasurable} to the function $|f|$. We obtain
  \begin{align}\label{itself}
     &\big\|[\psi ]_{\C^{1-\var}([s, t])}\big \|_{L_m(\Omega)}\nn\\
    &\,\,=\Bigl\|\int_s^t |f_r|(Y_{r} + \psi_r+x) \diff r\Bigr\|_{L_m(\Omega)}\nn\\
    &\,\, \le C \| f \|_{L_q([s, t], L_p(\R^d))}(t-s)^{1-\frac{1}{q}-\frac{Hd}{p}} \Bigl (1 + \big\|[\psi ]_{\C^{1-\var}([s, t])}\big \|_{L_m(\Omega)}^{1-\frac{1}{q}}(t-s)^{-H(1-\frac{1}{q})}\Bigr )\\
    &\,\,\le C \| f \|_{L_q([s, t], L_p(\R^d))} \Bigl (1 + \bigl \| [\psi ]_{\C^{1-\var}([s, t])} \bigr \|^{1-\frac{1}{q}}_{L_m(\Omega)} \Bigr),\nn
  \end{align}
  where $C=C(H,d,p,q,m)$ and in the last inequality we used that $1-\frac{1}{q}-\frac{Hd}{p}-H(1-\frac{1}{q})>0$ thanks to~\eqref{eq:main_cond}.
 By Lemma~\ref{lem:sqrbound}, this implies 
  \[
    \bigl \| [\psi ]_{\C^{1-\var}([s, t])} \bigr  \|_{L_m(\Omega)} \le C \bigl (\| f \|_{L_q([s, t], L_p(\R^d))}+ \| f \|_{L_q([s, t], L_p(\R^d))}^q \bigr ) \le C \bigl (1+ \| f \|_{L_q([s, t], L_p(\R^d))}^q \bigr )
  \]
  for a constant $C=C(H, d, p, q, m)$.
  Reinserting this bound into~\eqref{itself}, we deduce that there exists a constant $C=C(H,d,p,q,m)$ such that for all $(s,t)\in\Delta^2_{0,1}$, one has 
  \[
     \bigl \| [\psi ]_{\C^{1-\var}([s, t])} \bigr \|_{L_m(\Omega)} \le C \bigl (\| f \|_{L_q([s, t], L_p(\R^d))} + \| f \|_{L_q([s, t], L_p(\R^d))}^q\bigr ) (t-s)^{1-H-\frac{Hd}{p}-\frac{1-H}{q}},
  \]
  which is~\eqref{c313}.

Next, note that for any $m\ge1$ and $(s,t)\in\Delta^2_{0,1}$ we have 
\begin{equation*}
\| \psi_t-\psi_s\|_{L_m(\Omega)} \le 
\bigl \| [\psi ]_{\C^{1-\var}([s, t])} \bigr \|_{L_m(\Omega)} \le C \bigl (1+ \| f \|_{L_q([s, t], L_p(\R^d))}^q\bigr ) (t-s)^{1-H-\frac{Hd}{p}-\frac{1-H}{q}},
\end{equation*}
for $C=C(H, d, p, q, m)$. Therefore, by applying the quantitative version of Kolmogorov's continuity theorem  \cite[Theorem~A.11]{FV10} and the Chebyshev inequality, we get~\eqref{milshef}  with $\gamma=\frac12(1 - H -\frac{Hd}{p}-\frac{1-H}{q}) > 0$.
\end{proof}

Now, equipped with the crucial integral bound \eqref{3bound} and the a priori bound \eqref{c313}, we proceed to the final goal of this subsection: establishing stability results for the solutions of the SDE \Gref{x;f;Y}.

\begin{lemma}
\label{lem:drift_convergence}
Let $m \ge 2$ and $f \in L_{q}([0, 1], L_{p}(\R ^{d}))$ where $p, q \in [1, \infty)$ satisfy~\eqref{eq:main_cond}.  
Let $(x^n)_{n\in\N}$ be a sequence in $\R^d$ converging to $x$. 
Let $Y^n$ be a random element which has the same law as $Y$ and which is adapted to the filtration $\FF^n:=(\F^n_t)_{t\in[0,1]}$, and assume that $(Y^{n}, \FF^{n})$ satisfy Assumption~\ref{main_assumption}. 
Let $\psi$ and $\psi ^{n},\ n \in \N$ be measurable processes in $[0, 1] \times \Omega \to \R ^{d}$ adapted to $\FF$ and $\FF^{n}$ respectively and such that for some $C > 0$ we have for any $n \in \N$, 
\begin{equation}
\label{eq:unif_psin_psi}
  \| [\psi ^{n}]_{\C^{1-\var}([0, 1])}\|_{L_{m}(\Omega)} + \| [\psi]_{\C^{1-\var}([0, 1])}\|_{L_{m}(\Omega)} \le  C.
\end{equation}
Assume finally that for any $t \in [0, 1]$, 
\begin{equation}
 \lim_{n\to\infty }\psi_t^n = \psi_t\ \text{a.s.}\quad\text{and}\quad  
\lim_{n\to\infty }Y_t^n = Y_t\ \text{a.s.}
\end{equation}
Then for any $(s, t) \in \Delta ^{2}_{0, 1}$,
\begin{equation}\label{limstab}
 \lim_{n \to \infty } \Big \| \int_{s}^{t} f(x^{n} + \psi^{n}_{r} + Y^{n}_{r}) - f(x + \psi_{r} + Y_{r})\diff r \Big \|_{L_{m}(\Omega)} = 0.
\end{equation}
\end{lemma}

Note that if $f$ is continuous, the statement is obvious.  Lemma~\ref{lem:drift_convergence} shows that the regularization-by-noise effect of $Y$ is strong enough for the convergence \eqref{limstab} to occur even when $f$ is discontinuous.

\begin{proof}
Let $\varepsilon >0$ and let $g$ be a smooth and bounded function such that $ \| f - g  \|_{L_{q}([0, 1], L_{p}(\R ^{d}))} \le \varepsilon$. Define
\begin{equation}
 X := x + \psi + Y, \quad X^{n} := x^n + \psi ^{n} + Y^{n}, \quad n \in \N.
\end{equation}
Let $(s, t) \in \Delta_{0,1}^{2}$. We use the decomposition
\[
  \int_{s}^{t} |f(X^{n}_{r})-  f(X_{r})|\diff r \le 
\int_{s}^{t} |f-g|(X^{n}_{r})\diff r + \int_{s}^{t} |g(X^{n}_{r}) - g(X_{r})|\diff r + \int_{s}^{t} |f-g|(X_{r})\diff r.
\]
By Corollary~\ref{cor:onlymeasurable} and using \eqref{eq:unif_psin_psi}, we can bound the first and the last term with
\[
 \Big \| \int_{s}^{t} |f-g|(X^{n}_{r})\diff r \Big \|_{L_{m}(\Omega)} + \Big \| \int_{s}^{t} |f-g|(X_{r})\diff r \Big \|_{L_{m}(\Omega)}\le C \| f - g  \|_{L_{q}([0, 1], L_{p}(\R ^{d}))} \le C \varepsilon,
\]
where the constant does not depend on $n$.
By the hypothesis on $(x^{n})_{n}, (\psi^{n})_{n}$ and $(Y^{n})_{n}$, and $g$ being continuous, for any $r \in [0, 1]$,  $\lim_{n \to  \infty} |g(X^{n}_r) - g(X_r)| = 0$ almost surely. By the dominated convergence theorem, which is applicable because $g$ is bounded, we obtain
\[
  \lim_{n \to \infty }\Big \| \int_{s}^{t} |g(X^{n}_{r}) - g(X_{r})|\diff r  \Big \|_{L_{m}(\Omega)} \le  \lim_{n \to \infty}\int_{s}^{t}\| g(X^{n}_{r}) - g(X_{r}) \|_{L_{m}(\Omega)}\diff r = 0
\]
for the second term. Combining the two previous bounds, we get 
\[
\limsup_{n \to \infty}\Big \|  \int_{s}^{t} |f(X^{n}_{r})-  f(X_{r})|\diff r \Big \| \le C \varepsilon.
\]
Finally by taking $\varepsilon$ to $0$, we obtain the desired result.
\end{proof}

We use Lemma~\ref{lem:drift_convergence} to show that a limit of solutions to \Gref{x^n;b^n;Y^n} is a solution of the limiting equation. Again, if the drift were Hölder continuous, this would be obvious; here, however, we must rely on a priori bounds and the estimate \eqref{3bound}.

\begin{lemma}\label{lem:stability}
Assume that the parameters $p,q\in[1,\infty)$ satisfy~\eqref{eq:main_cond}. 	
Let $(b^{n})_{n\in\N}$ be a sequence of functions in $L_{q}([0, 1], L_{p}(\R ^{d}))$ converging to $b$ in $L_{q}([0, 1], L_{p}(\R ^{d}))$. 
Let $(x^n)_{n\in\N}$ be a sequence in $\R^d$ converging to $x$. 

Let $Y^n$ be a random element which has the same law as $Y$ and which is adapted to the filtration $\FF^n:=(\F^n_t)_{t\in[0,1]}$, and assume that $(Y^{n}, \FF^{n})$ satisfy Assumption~\ref{main_assumption}. 
Assume that $\psi^n$ is also adapted to  $\FF^n$ and solves \Gref{x^n;b^n;Y^n}. Moreover, suppose that for any $m\ge1$, $n\in\N$ we have
\begin{equation}\label{psivarnu}
  \| [\psi ^{n}]_{\C^{1-\var}([0, 1])}\|_{L_{m}(\Omega)}< \infty
\end{equation}
Suppose that there exists a measurable function $\psi\colon[0,1]\times\Omega\to\R^d$ adapted to the same filtration $\FF:=(\F_t)_{t\in[0,1]}$ as in Assumption~\ref{main_assumption} and such that for any $t \in [0, 1]$,
\begin{equation}\label{psiycond}
 \lim_{n\to\infty }\psi_t^n = \psi_t\ \text{a.s.}\quad\text{and}\quad  
\lim_{n\to\infty }Y_t^n = Y_t\ \text{a.s.}
\end{equation}
Then there exists a modification of $\psi$ which has almost surely continuous sample paths. Furthermore, any modification of $\psi$ which has almost surely continuous sample paths solves \Gref{x;b;Y} almost surely.
\end{lemma}
\begin{remark}Note that the bound \eqref{psivarnu} is not assumed to be uniform in $n$, and is therefore trivially satisfied if all the drifts $b^n$ are bounded.
\end{remark}

\begin{proof}[Proof of Lemma~\ref{lem:stability}]
Because $b^{n}$ converges to $b$, there exists $n_{0} \in \N$ such that for all $n \ge n_{0}$, 
\begin{equation*}
\| b^{n}\|_{L_{q}([0, 1], L_{p}(\R ^{d}))} \le \| b\|_{L_{q}([0, 1], L_{p}(\R ^{d}))} + 1 < +\infty.
\end{equation*}
Using this bound uniform in $n$ we now observe that by Corollary~\ref{cor:step02b} there exists a constant $\gamma >0$ such for any $m \ge 1$
\begin{equation}
	\label{1var_stabilityn}
	\| [\psi ^{n}]_{\C^{1-\var}([s, t])} \|_{L_{m}(\Omega)} \le C(t-s)^{\gamma},\quad (s, t) \in \Delta^{2}_{0, 1},\ n \ge n_{0},
\end{equation}
where $C = C(H, d, p, q, m) >0$.
By applying Lemma~\ref{lem:limit1var} with $m$ sufficiently large so that $\gamma > 1/m$, there exists a continuous modification of $\psi$ denoted further by $\wt\psi $. By Lemma~\ref{lem:limit1var}, it satisfies $ \lim_{n \to \infty} \psi^{n}_{t} = \wt \psi _{t}$ almost surely for any $t \in [0, 1]$, and moreover
\begin{equation}
	\label{1var_stability}
	\| [\wt \psi]_{\C^{1-\var}([s, t])} \|_{L_{m}(\Omega)} \le C(t-s)^{\gamma},\quad (s, t) \in \Delta^{2}_{0, 1}.
\end{equation}

%


Define for $n\in\N$
\begin{equation}\label{Xdef}
  X^{n} := x^n + \psi ^{n} + Y^{n}, \quad X := x + \wt \psi + Y.
\end{equation}
We want to show that $\wt \psi$ solves \Gref{x;b;Y}. First, let us verify that the integral in \Gref{x;b;Y} is well defined  by showing that $\int_{0}^{1} |b_{r}|(X_{r}) \diff r$ is finite almost surely. We already know, using \eqref{1var_stabilityn} and $\psi ^{n}$ solving \Gref{x^n;b^n;Y^n}, that $\int_{0}^{1} |b_{r}^{n}|(X^{n}_{r}) \diff r= [\psi ^{n}]_{\C^{1-\var}([0, 1])}$ is finite almost surely. We therefore bound the difference
\begin{align*}
  \Big  \| \int_{0}^{1} (|b_{r}^{n}|&(X^{n}_{r}) - |b_{r}|(X_{r})) \diff r  \Big  \|_{L_{2}(\Omega)} 
\\ &\le \Big  \| \int_{0}^{1} \big | |b_{r}^{n}| - |b_{r}| \big | (X^{n}_{r})\diff r  \Big  \|_{L_{2}(\Omega)}  
+ 
  \Big  \| \int_{0}^{1} |b_{r}|(X^{n}_{r}) - |b_{r}|(X_{r}) \diff r  \Big  \|_{L_{2}(\Omega)} 
\end{align*} 
To control the first term, we apply Corollary~\ref{cor:onlymeasurable} to the function $\big | |b_{r}^{n}| - |b_{r}| \big |$, recalling the bound \eqref{1var_stabilityn}. For the second term, we know that $X^{n}$ converges almost surely to $X$ because $\psi ^{n}$ converges to $\wt \psi$ almost surely and the hypothesis of the statement. We can therefore apply Lemma~\ref{lem:drift_convergence}. Combining the two bounds and using the convergence of $b^{n}$ to $b$, we obtain
\[
 \lim_{n \to \infty} \Big  \| \int_{0}^{1} |b_{r}^{n}|(X^{n}_{r}) - |b_{r}|(X_{r}) \diff r  \Big  \|_{L_{2}(\Omega)} = 0.
\]
By \eqref{1var_stabilityn}, this shows that 
\[
  \Big  \| \int_{0}^{1} |b_{r}|(X_{r}) \diff r  \Big  \|_{L_{2}(\Omega)} < +\infty,
\]
which in turns proves that almost surely, $\int_{0}^{\cdot}  b_r(x+\wt \psi_r+Y_r) \diff r$ is well defined and continuous. 

Now we verify that \Gref{x;b;Y} holds. Let $t \in [0, 1]$ and $n \ge n_{0}$ in $\N$. Consider the following decomposition:
\begin{align}\label{Aall}
	  \Big |  \wt \psi_t - \int_0^t b_r(X_r) \diff r \Big |& 
\le \Big | \wt \psi_t - \int_0^t b^n_r(X^n_r) \diff r \Big | 
+    \int_0^t |b^n_r - b_r|(X^n_r) \diff r \nn
\\& \quad + \int_0^t |b_r( X^n_r) - b_r(X_r)| \diff r
\end{align}
Let us bound all the terms in the right-hand side of the above inequality. 
By the assumptions of the lemma, $\psi^{n}_{t} = \int_{0}^{t} b_{r}^{n}(X_r^{n})\diff r$, and $\psi^{n}_{t}$ converges to $\wt \psi_{t}$ almost surely. Therefore,
\begin{equation}\label{A1}
\Big | \wt \psi_t - \int_0^t b^n_r(X^n_r) \diff r \Big | \quad \text{converges in probability to $0$ as $n \to \infty$}.
\end{equation}
For the second term, applying Corollary~\ref{cor:onlymeasurable} to the function $|b^n-b|$ and using the bound~\eqref{1var_stabilityn} for $m=2$, we get 
\begin{equation}\label{A2}
\Big \| \int_0^t |b^n_r - b_r|(X^{n}_{r}) \diff r \Big \|_{L_{2}(\Omega )}\le C \| b^n-b \|_{L_q([0, 1], L_p(\R^d))}
\end{equation}
for $C=C(H,d,p,q)>0$.
For the third term we apply Lemma~\ref{lem:drift_convergence} and 
\begin{equation}
\label{A3}
 \lim_{n \to \infty} \Big \| \int_{0}^{t} |b_{r}(X^{n}) - b_{r}(X_{r})| \diff r \Big \|_{L_{2}(\Omega)} = 0.
\end{equation}

Now we combine~\eqref{A1}--\eqref{A3}, substitute them into~\eqref{Aall}, and pass to the limit in probability as $n\to\infty$. We obtain that for any $\eps>0$,
\begin{equation*}
\P\Bigl(	\Big |  \wt \psi_t - \int_0^t b_r(X_r) \diff r \Big |>\eps\Bigr) = 0,
\end{equation*}
which proves that for any $t\in[0,1]$
\[
\P\Bigl(\wt \psi_t = \int_0^t b_r(x+\wt \psi_r+Y_r) \diff r \Bigr)=1,
\]
where we also used the definition of $X$ in~\eqref{Xdef}. 
The above inequality being true for any $t \in [0, 1]$ and because the rational numbers are countable, it  also holds that almost surely,
\begin{equation}
\label{eq:rational_times}
   \wt \psi_t = \int_0^t b_r(x+\wt \psi_r+Y_r) \diff r, \quad t \in  [0, 1] \cap \mathbb Q.
\end{equation}
Therefore, as we showed that the two sides of \ref{eq:rational_times} are continuous, almost surely,
\[
   \wt \psi_t = \int_0^t b_r(x+\wt \psi_r+Y_r) \diff r, \quad t \in  [0, 1].
\]
This means that $\wt \psi$ solves \Gref{x;b;Y} almost surely, which proves the lemma.
\end{proof}

\subsection{Verification of Assumption~\ref{main_assumption} for Gaussian and Lévy processes}

In this section, we prove that our main hypothesis on the noise, Assumption~\ref{main_assumption}, is satisfied for the fractional Brownian and Lévy noises.
We begin with the Gaussian case. 

For $d \in \N$, let $\I_d$ denote the $d \times d$ identity matrix. For two $d \times d$ matrices $A$ and $B$, we write $A \succeq B$ if $A - B$ is positive semidefinite.

\begin{lemma}
\label{lem:fbmyes}
Let $d\in\N$,  $H \in (0, 1)$. Let $Y\colon\Omega\times[0,1]\to\R^d$ be a centred Gaussian process adapted to the filtration $\FF=(\F_t)_{t\in[0,1]}$. Assume that for any $(s,u,t)\in\Delta^3_{0,1}$ we have
\begin{equation}
	\label{eq:indep}
\text{the random vector $\E^u Y_t-\E^s Y_t$ is Gaussian and independent of $\F_s$}.
\end{equation}
Suppose that for some $c>0$ we have 
\begin{equation}
	\label{eq:nearlynondeterminism}
	- \frac{d}{dr} \Var(Y_{t}\,|\,\F_{r})\succeq c (t-r)^{2H-1}\I_d, \quad 0 \le  r < t \le 1.
\end{equation}
Then $Y$ satisfies Assumption~\ref{main_assumption}.
\end{lemma}

\begin{proof}
It is immediate to see that Part (i) of Assumption~\ref{main_assumption} holds thanks to \eqref{eq:indep}.

To show part (ii) of Assumption~\ref{main_assumption}, we note that \eqref{eq:indep} implies  $\Var(\E^{u}Y_{t}) = \Var(\E^{s}Y_{t}) + \Var(\E^{u}Y_{t} -\E^{s}Y_{t})$ for any $(s,u,t)\in\Delta^3_{0,1}$. Therefore, we deduce  for any $(s,u,t)\in\Delta^3_{0,1}$
\begin{align}\label{gaussbound}
\Var (\E^{u}Y_{t} -\E^{s}Y_{t})&= \Var(\E^{u}Y_{t}) - \Var(\E^{s}Y_{t})\nn
\\ &= (\Var(Y_{t}) -  \Var(\E^{s}Y_{t})) - (\Var(Y_{t}) -  \Var(\E^{u}Y_{t}))\nn
\\ &= \Var(Y_{t} - \E^{s}Y_{t}) - \Var(Y_{t} - \E^{u}Y_{t})\nn
\\ &= - \int_{s}^{u} \frac{d}{dr}\Var(Y_{t}\,|\,\F_{r})\diff r\nn
\\ &\succeq  c \int_{s}^{u} (t-r)^{2H-1}\I_d\diff r\nn
\\ &= \frac{c}{2H} \bigl( (t-s)^{2H} - (t-u)^{2H} \bigr)\I_d,
\end{align}
where we used that $Y_{t} - \E^{s}Y_{t}$ is independent of $\F_{s}$.
Recall that we denoted by  $p_{s, u, t}$ the density of $\E^u Y_t-\E^s Y_t$. Since by \eqref{eq:indep} the vector $\E^u Y_t-\E^s Y_t$ is Gaussian, we obtain
\[
\| p_{s, u, t} \|_{L^{\infty}(\R^{d})} = C \det (\Var(\E^{u}Y_{s, t})) ^{-1/2} \le C \bigl( (t-s)^{2H} - (t-u)^{2H} \bigr)^{-d/2},
\]
for some $C=C(c,H)>0$. Here we used that if a matrix $A$ satisfies $A\succeq \I_d$, then $\det A\ge1$. This proves part (ii) of Assumption~\ref{main_assumption}.

We continue with part (iii) of Assumption~\ref{main_assumption}. Recall the notation $Y_{s,t}:=Y_t-\E^s Y_t$ for $(s,t)\in\Delta^2_{0,1}$. Let $\beta < 0$, $\rho \in [1, \infty ]$ and $f\colon\R ^{d}\to \R$ be a measurable function.  By the definition of the Besov norm for $\beta<0$
\begin{equation*}
\|f\|_{\B_\rho^\beta}=\sup_{a\in(0,1]}a^{-\frac\beta2}\|g_a* f\|_{L_{\rho}(\R^d)}=
\sup_{a>0}(a\wedge1)^{-\frac\beta2}\|g_a* f\|_{L_{\rho}(\R^d)},
\end{equation*}	
where $g_a$ stands for the $d$-dimensional Gaussian density of variance $a \I_d$. Now let $(s,t)\in\Delta_{0,1}^2$. By \eqref{gaussbound}, we have 
\begin{equation*}
\Sigma:=\Var(Y_{s, t})-\frac{c}{2H} (t-s)^{2H}\I_d\succeq0.
\end{equation*}
Therefore, we can define a centred Gaussian random vector $Z_1$ with variance matrix $\Sigma$. Let $Z_2$ be a centred Gaussian random vector with variance matrix $\frac{c}{2H}(t-s)^{2H}\I_d$ independent of $Z_1$. Then we have 
\begin{align*}
\|\E f(Y_{s,t}+\cdot)\|_{L_\rho(\R^d)}&=\|\E f(Z_1+Z_2+\cdot)\|_{L_\rho(\R^d)}\\
&\le \|\E f(Z_2+\cdot)\|_{L_\rho(\R^d)}\\
&=\|g_{\frac{c}{2H}(t-s)^{2H}}* f\|_{L_{\rho}(\R^d)}
\le C \|f\|_{\B_\rho^\beta}(t-s)^{\beta H}.
\end{align*}	
for $C=C(c,\beta,H)>0$, which is the desired bound~\eqref{eq:hyp_smoothing_nnew}. Thus part (iii) of Assumption~\ref{main_assumption} is proved.
\end{proof}

\begin{corollary}\label{c:318} Let $d\in\N$, $H\in(0,1)$,  $\FF=(\F_t)_{t\in[0,1]}$ is a filtration.
\begin{enumerate}[$($i$)$]
\item 
Let $Y\colon\Omega\times[0,1]\to\R^d$ be a centred Gaussian process. Assume that  $\FF$ is its natural filtration. Assume that $Y$ satisfies  condition \eqref{eq:nearlynondeterminism} for some $c>0$.
Then $Y$ satisfies Assumption~\ref{main_assumption}.

\item Let $Y=W^H$ be a $d$-dimensional $\FF$-fractional Brownian motion of Hurst parameter $H$. Then $Y$ satisfies Assumption~\ref{main_assumption}.
	
\item Let $Y$ be a $d$-dimensional Riemann-Liouville $\FF$-fractional Brownian motion of Hurst parameter $H$. Then $Y$ satisfies Assumption~\ref{main_assumption}.
\end{enumerate}
\end{corollary}
\begin{proof}
(i). Let  $(s, u,t)\in \Delta^{3}_{0, 1}$. The conditional expectation $\E^{s}Y_{t}$ is the orthogonal projection of $\E^u Y_{t}$ onto $L^{2}(\Omega, \F_{s}, \mathbb P)$. Therefore, $\E^u Y_{t}-\E^{s}Y_{t}$ is orthogonal to $L^{2}(\Omega, \F_{s}, \mathbb P)$. This implies that  $\E^{u} Y_{t} - \E^{s} Y_{t}$ is Gaussian and independent of $\F_{s}$ (recall that $Y$ is a Gaussian process and $\FF$ is its natural filtration). Thus, the process $Y$ satisfies condition \eqref{eq:indep}. Condition \eqref{eq:nearlynondeterminism} is satisfied by assumptions of the corollary. Hence, by Lemma \ref{lem:fbmyes}, $Y$ satisfies Assumption~\ref{main_assumption}.

(ii). It follows immediately from the representation of $W^{H}$ in Proposition~\ref{prop:representation}, that $W^H$ satisfies \eqref{eq:indep}. To verify \eqref{eq:nearlynondeterminism}, let $(r,t)\in\Delta^2_{0,1}$. 
By the representation of $W_{t}^{H}$ of Proposition~\ref{prop:representation},
\[
\Var(W^{H}_{t}\,|\,\F_{r}) = \Var(W^{H}_{t} - \E^{r}W^{H}_{t}) = \Var \Bigl( \int_{r}^{t}K_{H}(t, u) \diff B_{u}\Bigr) = \int_{r}^{t} K_{H}(t, u)^{2} \diff u\cdot \I_d.
\]
Therefore, by using Proposition~\ref{prop:representation}\ref{part1fbm},
\[
- \frac{d}{dr}\Var(W^{H}_{t}\,|\,\F_{r}) = K_{H}(t, r)^{2}\I_d \ge C (t-r)^{2H-1}\I_d,
\]
for $C=C(H,d)>0$,
which proves \eqref{eq:nearlynondeterminism}. Therefore, by Lemma \ref{lem:fbmyes}, Assumption~\ref{main_assumption} holds.

(iii). Recall that if $Y$ is a $d$-dimensional Riemann-Liouville $\FF$-fractional Brownian motion, then 
there exists an $\FF$-Brownian motion $B$ such that for some $C>0$ we have 
	\begin{equation*}
Y_t=C\int_0^t (t-s)^{H-\frac12}\,dB_s,\quad t\in[0,1].
\end{equation*}
This implies that $Y$ satisfies \eqref{eq:indep} and for any $(r,t)\in\Delta^2_{0,1}$
\[
- \frac{d}{dr}\Var(W^{H}_{t}\,|\,\F_{r}) = C  (t-r)^{2H-1}\I_d.
\]
Thus  \eqref{eq:nearlynondeterminism} is also satisfied and by Lemma \ref{lem:fbmyes}, Assumption~\ref{main_assumption} holds.
\end{proof}

Next, let us verify Assumption~\ref{main_assumption} for a Lévy processes satisfying structural condition
 \eqref{eq:characteristic_bound}.
 
\begin{lemma}
\label{lem:assumption_levy}
Let $\FF=(\F_t)_{t\in[0,1]}$ be a complete filtration and let $L$ be a $d$-dimensional $\FF$-Lévy process satisfying \eqref{eq:characteristic_bound} for some $\alpha \in (1, 2)$.
Then $L$ satisfies Assumption~\ref{main_assumption} with this filtration $\FF$ and $H := 1/\alpha$.
\end{lemma}

\begin{proof}
Since $L$ is a Markov process, 
\begin{equation}\label{ystlevy}
  Y_{s, t} := L_{t} - \E ^{s}L_{t} = L_{t} - L_{s}, \quad (s, t)\in \Delta ^{2}_{0, 1}.
\end{equation}
Therefore, $Y_{s, t}$ is independent of $\F_{s}$, which proves part (i) of Assumption~\ref{main_assumption}.

For any $t >0$, define $p_{t}$ the law of $L_{t}$.  
By the Fourier inversion theorem, the lower bound on the characteristic exponent \eqref{eq:characteristic_bound}, and a change of variables in the last integral, we obtain
\begin{align}
\label{eq:inverse_fourier}
\| p_t \|_{L^\infty(\R^{d})} &\le \| \hat p_t \|_{L^1(\R^{d})}
\le \int_{\R^d} e^{-t \Re \Phi(\lambda )} \diff \lambda 
\le \int_{\R^d} e^{-c_{1} t |\lambda |^\alpha} \diff \lambda +\int_{|\lambda|\le N} 1 \diff\lambda \nn\\
&= C t^{-d / \alpha}
\end{align}
for $C=C(\alpha,c_1,d,N)$.
Now let $(s, u, t) \in \Delta ^{3}_{0, 1}$. Then 
\[ 
p_{s, u, t} =\Law(\E^{u} Y_{s,t})=\Law(L_u-L_s) = \Law(L_{u-s}) = p_{u-s}.
\] 
By applying \eqref{eq:inverse_fourier}, this proves part (ii) of Assumption~\ref{main_assumption}.

Finally, we verify part~(iii). Let $\rho \in [1, \infty ]$, $n \in \N$, and $f : \R ^{d}\to \R$ measurable. Using Young's convolution inequality for Besov spaces \cite[Theorem~2.2]{KS21}, we obtain for any $t \in (0, 1]$,
\begin{equation}
\label{eq:conv_besov}
\| f * p_{t} \|_{L_{\rho }(\R ^{d})} 
\le C\| f * p_{t} \|_{\B_{\rho, 1}^{0}(\R ^{d})} 
\le C\| f \|_{\B_{\rho, 1}^{-n}(\R ^{d})} \| p_{t} \|_{\B_{1, \infty}^{n}(\R ^{d})},
\end{equation}
for $C=C(\rho,d,n)$. 
Note by $(\mathcal P_{t})_{t \ge 0}$ the semigroup associated to $L$. By \cite[Proposition 2.8]{BDG24}, condition  \eqref{eq:characteristic_bound} implies  for any bounded measurable function $\phi : \R^{d} \to \R$,
\begin{equation*}
\| \nabla \mathcal P_{t} \phi  \|_{L^{\infty}(\R^{d})} \le C t^{-1/\alpha} \| \phi  \|_{L^{\infty }(\R^{d})}, \quad t \in (0, 1].
\end{equation*}
for $C=C(\alpha,d)$.
Let $\mu $ denote a multi-index such that $|\mu | = n$, and let $x_{i} \in 1\dots d$ for $i \in 1\dots n$ be such that $\partial^{\mu} = \partial_{x_{1}}\dots \partial_{x_{n}}$.
We remark that the derivation commutes with $\mathcal P_{t/n}$, so $\partial^{\mu}\mathcal P_{t} = (\partial_{x_{1}}\mathcal P_{t/n})\dots (\partial_{x_{n}}\mathcal P_{t/n})$.
We can therefore iterate the previous bound $n$ times and we obtain 
\begin{equation}
\label{eq:H1}
\| \partial^{\mu} \mathcal P_{t} \phi  \|_{L^{\infty}(\R^{d})} \le  C^{n} (t/n)^{-n/\alpha }\| \phi  \|_{L^{\infty }(\R^{d})} = C t^{-n/\alpha}\| \phi  \|_{L^{\infty }(\R^{d})}, \quad t \in (0, 1]
\end{equation} 
for $C=C(\alpha,d,n)$.
Because of \eqref{eq:characteristic_bound}, the characteristic exponent $\Phi$ clearly satisfies 
\[
\lim_{|\lambda| \to \infty } \frac{\Re \Phi  (\lambda) }{\log (1 + |\lambda|)} = \infty,
\]
so by \cite[Theorem 1 (a, b)]{Knopova}, the density $p_{t}$ is in $C^{\infty}(\R^{d})$ for any $t>0$.

Now we want to show that the bound \eqref{eq:H1} implies 
\begin{equation}
\label{eq:sobolev_bound_density}
\| \partial^{\mu} p_{t} \|_{L^{1}(\R ^{d})}  \le  C t^{-n/\alpha}, \quad t \in (0, 1],
\end{equation}
which we will do by adapting the argument of \cite[Lemma~4.1]{KS19} valid for $n = 1$. 
Let $\chi_{m}$ for any $m \in \N$ be a smooth cut-off function with $\1_{B(0, m)} \le \chi_{m} \le \1_{B(0, m+1)}$.
Fix $m \in \N$, $x \in \R ^{d}$ and $t \in (0, 1]$. Recall that the operator $\mathcal P_{t}$ is just the convolution with the smooth function $p_{t}$, 
\[
  \mathcal P_{t} (\phi \chi _{m})(x)  = \E [(\phi \chi _{m})(x + L_{t})] = \int_{\R ^{d}}(\phi\chi _{m})(x+y)p_{t}(y)\diff y.
\]
 By differentiating under the integral (see \cite[Proposition~A.1]{KS19}),
\begin{equation}
\label{eq:diff_under_integ}
  \partial ^{\mu }\mathcal P_{t}(\phi \chi_{m})(x) = (-1)^{n}\int_{\R ^{d}} (\phi\chi _{m})(x+y)\partial ^{\mu }p_{t}(y) \diff y.
\end{equation}
The idea is to apply the following duality relation, where the supremum is taken over the bounded and measurable functions:
\begin{align*}
\int_{|y| \le m} |\partial ^{\mu }p_{t}(y)|\diff y 
&= \sup_{\phi} \Bigl\{ \Big| \int_{|y| \le m} \phi(y) \partial ^{\mu }p_{t}(y)\diff y\Big|;\ \| \phi \|_{L^{\infty}(\R^{d})} \le 1 \Bigr\}  
\\&= \sup_{\phi} \Bigl\{ \Big| \int_{|y| \le m} (\chi _{m}\phi)(y) \partial ^{\mu}p_{t}(y)\diff y\Big|;\ \| \phi \|_{L^{\infty}(\R^{d})} \le 1 \Bigr\}  
\\ &\le  \sup_{\phi} \Bigl\{ \|  \partial ^{\mu }\mathcal P_{t}(\phi \chi_{m})\|_{L^{\infty}(\R^{d})};\ \| \phi \|_{L^{\infty}(\R^{d})} \le 1 \Bigr\}  
\\ &\le C t^{-n/\alpha},
\end{align*}
for $C=C(\alpha,d,n)$, where we also used the definition of $\chi_{m}$, \eqref{eq:diff_under_integ} at $x=0$, and \eqref{eq:H1}. 
As the previous bound is valid for any $m \in \N$, we can pass to the limit as $m$ goes to infinity, and this proves \eqref{eq:sobolev_bound_density}.

The multi-index $\mu $ being arbitrary, we immediately deduce from \eqref{eq:sobolev_bound_density} that because $t \in (0, 1]$, 
\begin{equation}
\label{eq:density_sobolev}
\| p_{t} \|_{\B_{1, \infty}^{n}(\R ^{d})} 
\le C \| p_{t} \|_{W^{n,1}(\R ^{d})}
\le C (1 + \sup_{|\mu| = n}\| \partial ^{\mu } p_{t} \|_{L^{1}(\R ^{d})})
\le C t^{-n/\alpha}.
\end{equation}
Combining \eqref{eq:conv_besov} and \eqref{eq:density_sobolev} we obtain for $t \in (0, 1]$,
\begin{equation}
\label{eq:space1}
\| f * p_{t} \|_{L_{\rho }(\R ^{d})} \le C t^{-n/\alpha} \| f \|_{\B_{\rho, 1}^{-n}(\R ^{d})},
\end{equation}
for $C=C(\alpha,\rho,d,n)$.
Using Young's convolution inequality, we also have for $t \in (0, 1]$,
\begin{equation}
\label{eq:space0}
\| f * p_{t} \|_{L_{\rho }(\R ^{d})} 
\le \| f \|_{L_{\rho }(\R ^{d})} 
\le C \| f \|_{\B^{0}_{\rho, 1}(\R ^{d})}.
\end{equation}
Note that the constants in \eqref{eq:space1} and \eqref{eq:space0} do not depend on $f$ and $t$. The map $f \mapsto f * p_{t}$ being linear, we are with \eqref{eq:space1} and \eqref{eq:space0} in the framework of the real interpolation theory. Let $\theta \in (0, 1)$. We have from \cite[Theorem~4.25]{Saw18} the interpolation space
\[
  (\B^{0}_{\rho, 1}(\R ^{d}), \B_{\rho, 1}^{-n}(\R ^{d}))_{\theta, \infty} = \B^{-n\theta}_{\rho, \infty}.
\]
By the real interpolation theorem \cite[Theorem~4.24]{Saw18}, we have an (exact) interpolation pair of exponent $\theta$ and
\begin{equation}
\label{eq:interpolated}
\| f * p_{t} \|_{L_{\rho }(\R ^{d})} \le C t^{-n\theta/\alpha}\| f \|_{\B^{-n\theta}_{\rho}} \quad t \in (0, 1].
\end{equation}
for  $C=C(\alpha,d,n)$.
Finally, let $\beta < 0$ and $(s, t) \in \Delta_{0, 1}^{2}$. Remember that $\Law(Y_{s, t}) = p_{t-s}$. By choosing $n \in \N$ such that $-n < \beta$ and  $\theta = -\frac{\beta }{n} \in (0, 1)$ in \eqref{eq:interpolated}, we obtain
\[
\| \E f(Y_{s, t} + \cdot) \|_{L_{\rho}(\R^{d})} 
= \| f * p_{t-s} \|_{L_{\rho }(\R ^{d})} 
\le C (t-s)^{\frac\beta\alpha}\| f \|_{\B^{\beta }_{\rho}},
\]
for $C=C(\alpha,\beta,\rho,d)$,
which proves part (iii) of Assumption~\ref{main_assumption}. 
\end{proof}

\subsection{Proofs of weak existence}

Finally, we are ready to show weak existence of solutions to the SDEs~\eqref{eq:main_eq} and~\eqref{eq:main_eq_levy}. The following simple argument allows us to focus on the case where $q < \infty$.
\begin{remark}\label{r:qinf}
  The case $q=\infty$ in Theorem~\ref{th:weak_existence} and Theorem~\ref{th:levy_existence} follows from the case $q \in [1, \infty)$. Indeed, let $d$, $H$, $p$, $q$, $x$, $T$ and $b$ be as in Theorem~\ref{th:weak_existence}, satisfying~\eqref{eq:main_cond} with $q = \infty$. Since the inequality of condition~\eqref{eq:main_cond} is strict, there exists $\wt q \in [1, \infty)$ such that~\eqref{eq:main_cond} with $\wt q$ instead of $q$ is still satisfied. Then $b \in L_{\wt q}([0, T], L_p(\R^d))$ and by applying (i) with $\wt q$ instead of $q$, the SDE~\eqref{eq:main_eq} has a weak solution. Similarly, for (ii), applying the result for $\wt q$ gives the result for $q=\infty$. The same can be done for Theorem~\ref{th:levy_existence}.
\end{remark}
Therefore, without loss of generality, we consider only the case $q<\infty$ in the following.

The scheme of the proofs of Theorem~\ref{th:weak_existence} and Theorem~\ref{th:levy_existence} is quite standard once the bounds from the previous sections have been obtained; see, e.g., \cite{BLM23}, \cite[Section~8]{GG23}, though some modifications are needed to handle the L\'evy case. We consider a sequence of solutions to the regularized equation \eqref{eq:approx}, and Corollary~\ref{cor:step02b} allows us to establish the tightness of this sequence. After extracting a  convergent subsequence and verifying that the limit noise is the one we are interested in, we apply the stability result Lemma~\ref{lem:stability} to conclude that the limit of our subsequence is a weak solution to~\eqref{eq:main_eq}.

\begin{proof}[Proof of Theorem~\ref{th:weak_existence}]
We first remark that (i) follows straightforwardly from (ii). Indeed, such a sequence of smooth functions $b^n \to b$ exists by a density argument, and on any probability space containing a fractional Brownian motion $W^H$, we can construct, by Banach's fixed-point theorem, the strong solution $X^n$ to (\ref{eq:approx}) that is adapted to the natural filtration of $W^H$.
Therefore, the conditions of (ii) are satisfied, which implies  weak existence.

Thus, we concentrate on (ii). 

\textbf{Step 1: setup.} We have a probability space $(\Omega, \F, \P)$ with a fractional Brownian motion $W^{H}$. We denote by $\FF = (\F_{t})_{t\in[0,1]}$ the natural filtration associated with $W^{H}$. By assumption, we have a sequence $\{ b_{n}, n \in\N \}$ of smooth functions converging to $b$ and strong solutions $X^{n}$ to (\ref{eq:approx}). The solutions $X^{n}$ are therefore adapted to $\FF$.  Without loss of generality, we assume $\sup_{n\in\N} \|b^n\|_{L_q([0,1],L_p(\R^d))}\le 2
\|b\|_{L_q([0,1],L_p(\R^d))}$.

We define the drift $\psi^n$ by
  \begin{equation}
    \label{defdrift} \psi^n_t := \int_0^t b_r^n(X_r^n) \diff r, \quad t \in [0, 1]
  \end{equation}
  when equation (\ref{eq:approx}) is satisfied and $0$ otherwise.
  Therefore, by (\ref{eq:approx}) we have almost surely
 \begin{equation}
    \label{eq:sum} X^n_t = x^n + \psi_t^n+ W_t^{H} \quad \text{for all $t$ in $[0, 1]$}.
  \end{equation}
Hence, $\psi^n$ solves \Gref{x^n;b^n;W^H}.

  \textbf{Step 2: tightness.}
  Take $\theta \in [0, \frac{1}{2})$ such that $\theta > H - \frac{1}{2}$ and define, using Proposition~\ref{prop:representation} the Brownian motion $B := \Phi (W^{H})$. 
It is easy to see that the sequence
  \begin{equation}
    \label{tight} \{\Law(\psi^n, W^{H}, B), n \in\N\}\ \text{is tight in}\ C([0, 1], \R^{d})^{2} \times \C^{\theta}([0, 1], \R^d).
  \end{equation}
More precisely, we use Lemma~\ref{lem:fbmyes} to show that Assumption~\ref{main_assumption} is satisfied for $W^H$ and $\FF$. Because $b^{n}$ is bounded by hypothesis, $[\psi^{n}]_{\C^{1-\var}([0, 1])} \le \| b^{n} \|_{L_{\infty}([0, 1]\times \R^d)} < \infty$ for all $n$. We can therefore apply Corollary~\ref{cor:step02b} and the tightness follows using this bound.
We refer the reader to the arguments of \cite[Lemma 5.2]{BLM23} for more details on this. 
%

  \textbf{Step 3: convergence.}
  Because $\C ([0, 1], \R^{d})^{2} \times \C^{\theta}([0, 1], \R^d)$ is a complete separable metric space there exists by Prokhorov's theorem a subsequence of (\ref{tight}) that converges in law to a probability measure.
  By Skorokhod's representation theorem, there exists another probability space $(\wt \Omega , \wt \F , \wt \P)$, random variables $\wt \psi $,  $\wt W^H$, $\wt B$, and for each $n \in\N$, random variables $\wt \psi^n$, $\wt W^{n, H}$, $\wt B^n$ such that
  \begin{equation}
    \label{skoeq} \Law(\wt \psi^n,  \wt W^{H, n}, \wt B^n) = \Law(\psi^n,  W^H, B),
  \end{equation}
  and such that a subsequence of
  \begin{equation}
    \label{skoconv} (\wt \psi^n, \wt W^{H, n},\wt B^n)\ \text{converges a.s.\ to}\ (\wt \psi,  \wt W^H, \wt B)\ \text{in}\ C ([0, 1], \R^{2d}) \times \C^{\theta}([0, 1], \R^d).
  \end{equation}
  To avoid overburdening the proof with indices, we again denote one of these subsequences by $(\wt \psi^n,  \wt W^{H, n},\wt  B^n)$.
    The limit $(\wt X:=x+\wt\psi+\wt W^H, \wt W^H)$ is our candidate for a solution.

  \textbf{Step 4: the limit noise is a $\wt \FF$--fractional Brownian motion.}
  Introduce the filtrations $\wt \FF = (\wt\F_t)_{t\in[0,1]}$, $\wt \FF^n = (\wt\F^n_t)_{t\in[0,1]}$, $n\in\N$, on $(\wt \Omega, \wt \F, \wt \P)$
  \begin{equation}\label{filtr}
    \wt \F_t := \sigma (\wt \psi_s, \wt B_s\ \text{for}\ s\in[0,t]),\qquad 
    \wt \F^n_t := \sigma (\wt B^n_s\ \text{for}\ s\in[0,t]).
  \end{equation}
%
We claim that $\wt B$ is a $\wt \FF$--Brownian motion and that $\wt W^H=\Psi(\wt B)$, see Proposition~\ref{prop:representation}, is a $\wt \FF$--fractional Brownian motion. This result can be found in the proof of \cite[Corollary 5.4]{BLM23}.

  \textbf{Step 5: the limit is a solution.}
We would like to apply Lemma~\ref{lem:stability}. 
By the above, $\wt W^{H,n}$ is a $\wt\FF^n$--fractional Brownian motion and $\wt W^{H}$ is a $\wt\FF$--fractional Brownian motion. Therefore, by Lemma~\ref{lem:fbmyes},
the pairs $(\wt W^{H,n}, \wt\FF^n)$  and $(\wt W^{H}, \wt\FF)$ satisfy Assumption~\ref{main_assumption}. 
Next, we see that by definition the drift $\wt \psi^{n}$ solves \Gref{x^n;b^n;\wt W^{H, n}} and is adapted to $\wt\FF^n$, $b^{n}$ is bounded. By construction, $\wt \psi$ is adapted to $\wt\FF$.
By~\eqref{skoconv}, $\wt \psi^{n} \to \wt \psi$ and $\wt W^{H, n} \to \wt W^{H}$ almost surely in $C ([0, 1], \R ^{d})$. Hence, condition~\eqref{psiycond} holds.

Thus, all the conditions of Lemma~\ref{lem:stability} are satisfied and we have 
\begin{equation}\label{limweak}
\P\Bigl(\forall t\in[0,1]:\, \wt \psi_t=\int_0^t b_r(x+\wt \psi_r+\wt W^H_r)\,dr\Bigr)=1.
\end{equation}

Put now $\wt X_t:=x+\wt \psi_t+\wt W_t^H$, $t\in[0,1]$. We claim that $(\wt X_t,\wt W^H)$ is a weak solution to~\eqref{eq:main_eq}. Indeed, by Step~4, $\wt W^H$ is a $\wt \FF$--fractional Brownian motion and $\wt X$ is adapted to $\wt \FF$. The identity~\eqref{limweak} implies that $(\wt X,\wt W^H)$ satisfy~\eqref{eq:main_eq}, and thus,  $(\wt X_t,\wt W^H)$ is indeed a weak solution to~\eqref{eq:main_eq}.
\end{proof}

The proof of weak existence in the Lévy case closely follows the proof of Theorem~\ref{th:weak_existence} with two minor differences. On the one hand, it is easier to demonstrate that the limit $L$ of the noise is still a Lévy process. Indeed, since $L$ is Markovian, we do not need to consider its representation with respect to an additional auxiliary Markov process. On the other hand, the Lévy noise is not continuous, so we have to work in the Skorokhod space $\D([0,1], \R^d)$ instead of $C([0,1], \R^{d})$.

\begin{proof}[Proof of Theorem~\ref{th:levy_existence}]
 As in the proof of Theorem~\ref{th:weak_existence}, part (i) follows immediately from part (ii). Therefore, we only prove (ii).

\textbf{Step 1: setup.} 
 We are given a probability space $(\Omega, \F, \P)$ with a Lévy process $L$ satisfying \eqref{eq:characteristic_bound}. We denote by $\FF = (\F_{t})_{t}$ the natural filtration associated with $L$. By hypothesis, we have a sequence $\{ b_{n}, n \ge 1 \}$ of smooth functions converging to $b$ and strong solutions $X^{n}$ to (\ref{eq:approx_levy}). The solutions $X^{n}$ are therefore adapted to $\FF$.  
We define the drift $\psi^n$ by
  \begin{equation}
    \label{defdrift_levy} \psi^n_t := \int_0^t b_r^n(X_r^n) \diff r, \quad t \in [0, 1]
  \end{equation}
  when equation (\ref{eq:approx}) is satisfied and $0$ otherwise.
  Therefore by (\ref{eq:approx}), we have almost surely
 \begin{equation}
    \label{eq:sum_levy} X^n_t = x_0 + \psi_t^n+ L_t \quad \text{for all $t$ in $[0, 1]$}.
  \end{equation}
Thus, $\psi^n$ solves \Gref{x^n;b^n;L}.

 \textbf{Step 2: tightness.} 
The main hypothesis on the noise Assumption~\ref{main_assumption} is satisfied by Lemma~\ref{lem:assumption_levy}.
Exactly as for Theorem~\ref{th:weak_existence}, the sequence of laws
  \begin{equation}
    \label{tight_levy} \{\Law(\psi^n, L), n \ge 1\}\ \text{is tight in}\ C([0, 1], \R^d) \times \D([0, 1], \R^d).
  \end{equation}
The difference is that here the noise lives in the Skorokhod space $\D([0, 1], \R^d)$, which we recall is Polish by \cite[Theorem~VI.1.14(a)]{JS87}. We still refer the reader to the arguments of \cite[Lemma 5.2]{BLM23} showing the tightness from Corollary~\ref{cor:step02b}.

Tightness of the sequence $(X^n, L)_{n\in\N}$ in $\D([0, 1], \R^{2d})$ follows from~\eqref{tight_levy} and the fact that $\psi ^{n}$ is continuous, see \cite[Proposition~VI.1.23]{JS87}.

%

  \textbf{Step 3: convergence.}
  Because $C([0, 1], \R^d) \times \D([0, 1], \R^d)$ is a complete separable metric space there exists by Prokhorov's theorem a subsequence of (\ref{tight_levy}) that converges in law to a probability measure. Using again Skorokhod's representation theorem, we see that there exists a probability space $(\wt \Omega , \wt \F , \wt \P)$, random variables $\wt \psi $, $\wt L$ and for each $n \ge 1$, random variables $\wt \psi^n$ and $\wt L^n$ such that $\Law(\wt \psi^n, \wt L^n) = \Law(\psi^n, L),$ and such that a subsequence of $(\wt \psi^n, \wt L^n)$ converges almost surely to $(\wt \psi, \wt L)$ in $C([0, 1], \R^d) \times \D([0, 1], \R^d)$. As before,  we denote one of these subsequences again by $(\wt \psi^n, \wt L^n)$.

  \textbf{Step 4: the limit noise is an $\wt \FF$-Lévy process with the same law.}
  Clearly, $\Law(\wt L) = \Law(\wt L^{n})$. Therefore, by Proposition~\ref{prop:skorokhod}, $\wt L$ is a Lévy process with the same law as $L$.
  We introduce the filtrations $\wt \FF = (\wt\F_t)_{t\in[0,1]}$, $\wt \FF^n = (\wt\F^n_t)_{t\in[0,1]}$, $n\in\N$, on $(\wt \Omega, \wt \F, \wt \P)$
  \begin{equation}\label{filtrl}
  	\wt \F_t := \sigma (\wt \psi_s, \wt L_s\ \text{for}\ s\in[0,t]),\qquad 
  	\wt \F^n_t := \sigma (\wt L^n_s\ \text{for}\ s\in[0,t]).
  \end{equation}  
  By Proposition~\ref{prop:skorokhod}, the convergence of $\wt L^n$ to $\wt L$ in  $\D([0, 1], \R^{d})$ implies the convergence of the finite dimensional distributions of $\wt L^n$ to $\wt L$.
  The same reasoning done in \cite[Corollary 5.4]{BLM23} for a Brownian motion proves that for $(s, t) \in \Delta^2_{0, 1}$, $\wt L_t - \wt L_s$ is independent of $\wt \F_s$. Therefore, $(\wt L_t)_{t\in[0,1]}$ is an $\wt \FF$-Lévy process with the same law as $L$.

 \textbf{Step 5: the limit is a solution.}
 Again, we would like to apply Lemma~\ref{lem:stability}. 
By the above, $\wt L^{n}$ is an $\wt\FF^n$-Lévy process and $\wt L$ is an $\wt\FF$-Lévy process, both with the law of $L$. Therefore, by Lemma~\ref{lem:assumption_levy},
the pairs $(\wt L^{n}, \wt\FF^n)$  and $(\wt L, \wt\FF)$ satisfy Assumption~\ref{main_assumption}. 
By construction, the drift $\wt \psi^{n}$ solves \Gref{x^n;b^n;\wt L^{n}} and is adapted to $\wt\FF^n$, and $b^{n}$ is bounded. By definition, $\wt \psi$ is adapted to $\wt\FF$.
By Step~3, $\wt \psi^{n} \to \wt \psi$ almost surely in $C ([0, 1], \R ^{d})$ and $\wt L^{n} \to \wt L$ almost surely in $\D ([0, 1], \R ^{d})$, which implies by Proposition~\ref{prop:skorokhod} the almost sure convergence of $\wt L^n_t$ to  $\wt L_t$ at any fixed $t\in[0,1]$. Hence, the condition~\eqref{psiycond} holds.

Thus, all the conditions of Lemma~\ref{lem:stability} are satisfied and we have 
\begin{equation*}
	\P\Bigl(\forall t\in[0,1]:\, \wt \psi_t=\int_0^t b_r(x+\wt \psi_r+\wt L_r)\,dr\Bigr)=1.
\end{equation*}
Putting again $\wt X_t:=x+\wt \psi_t+\wt L_t$, $t\in[0,1]$, we see from the above equation that $(\wt X_t,\wt L)$ is a weak solution to~\eqref{eq:main_eq_levy}.
\end{proof}

\subsection{Proof of strong existence}

We now turn to the proofs of Theorems~\ref{th:strong_existence} and~\ref{th:str_levy_existence}, which establish strong existence in dimension $1$. The main idea is as follows. In the proofs of Theorems~\ref{th:weak_existence} and~\ref{th:levy_existence} we constructed a sequence of solutions to regularized SDEs that converges weakly, which implied weak existence. It turns out that if $d=1$, one can actually  construct a sequence of solutions to regularized SDEs that converges almost surely, which yields strong existence. This construction is based on the comparison principle. The idea goes back to  Gy\"ongy and Pardoux \cite{GP93a, GP93b}; however, there is an important difference here. In \cite{GP93a}, the boundedness of the drift was assumed, while \cite{GP93b} relied on the Girsanov theorem. In our setting, the Girsanov theorem is not available and the drifts may be unbounded; therefore, we instead rely on the a priori bound \eqref{c313} and the Krylov-type bound \eqref{3bound}.

As in the case of weak existence, we can, without loss of generality, assume that $q \in [1, \infty)$.

\begin{proof}[Proof of Theorems \ref{th:strong_existence} and \ref{th:str_levy_existence}.]
Let $(\Omega, \F, \P)$ be a probability space. To cover both the case of fractional and Lévy noise, we define $Y : \Omega \times [0, 1] \to \R$ to be a stochastic process satisfying Assumption~\ref{main_assumption} with $\FF = (\F_{t})_{t \in [0, 1]}$ being the natural filtration of $Y$. 
We recall that by Lemmas~\ref{lem:fbmyes} and~\ref{lem:assumption_levy}, the processes from Theorems \ref{th:strong_existence} and~\ref{th:str_levy_existence} both satisfy Assumption~\ref{main_assumption}.

Because $p, q < \infty$, we can construct a sequence $(b^{n})_{n \in \N}$  of smooth and bounded functions converging to $b$ in $L_{q}([0, 1], L_{p}(\R))$. Moreover, up to taking a subsequence we can suppose that 
\begin{equation}
\label{eq:fast_lqp}
\forall n \in \N,\ \| b^{n} - b \|_{L_{q}([0, 1], L_{p}(\R))} \le 2^{-n}
\quad \text{and} \quad \lim_{n \to \infty}b^{n} = b\ \text{almost everywhere}.
\end{equation}
We remark that because $\sup_{k \ge 0} |b^{k}| \le |b| + \sup_{k \ge 0}|b^{k} - b|$, we can deduce from \eqref{eq:fast_lqp} that 
\begin{align}
\label{eq:uniform_sup_bk}
\| \sup_{k \ge 0} |b^{k}| \|_{L_{q}([0, 1], L_{p}(\R))} 
&\le \| b \|_{L_{q}([0, 1], L_{p}(\R))} + \sum_{k =0}^{\infty } \| b^{k} - b \|_{L_{q}([0, 1], L_{p}(\R))}\nonumber
\\ &\le \| b \|_{L_{q}([0, 1], L_{p}(\R))} + 2 < \infty.
\end{align}

\textbf{Step 1.}
First, we construct an appropriate sequence of smooth drifts that converges to $b$ in $L_{q}([0, 1], L_{p}(\R))$. As in \cite{GP93a, GP93b}, define for any integers $ 0 \le n \le k$
\begin{equation*}
  f^{n,k} := \min(b^{n}, \dots, b^{k})\quad \text{and}\quad f^{n,\infty} := \inf_{k \ge n}f^{n, k} = \inf _{k \ge n} b^{k}.
\end{equation*}

Using \eqref{eq:uniform_sup_bk}, we get for any $0\le n \le k$
\begin{equation}\label{fnkbound}
	\| f^{n, k} \|_{L_{q}([0, 1], L_{p}(\R))} \le \| \sup_{j \ge 0} |b^{j}| \|_{L_{q}([0, 1], L_{p}(\R))}\le 
	\|b \|_{L_{q}([0, 1], L_{p}(\R))}+2.
\end{equation}

Next, we claim that for any $0 \le n \le k$, $x\in\R$ we have 
\begin{equation}\label{claimstrong}
|f^{n, k}(x) - f^{n, \infty}(x)|\le|b^{k}(x)-b(x)| + \sum_{j=k+1}^{\infty}|b^{j}(x)-b(x)|.
\end{equation}
Indeed, if $\inf_{j \ge n} b^{j}(x)=b^i(x)$ for some $i\le k$, then $f^{n, k}(x)=f^{n, \infty}(x)=b^i(x)$ and the claim trivially holds. Otherwise,  $f^{n, \infty}(x)=f^{k+1, \infty}(x)$ and 
\begin{align*}
 |f^{n, k}(x) - f^{n, \infty}(x)| &= f^{n, k}(x) - f^{n, \infty}(x)\le b^k(x)-f^{k+1, \infty}(x)\\
 &\le   (b^k(x)-b(x))-(f^{k+1, \infty}(x)-b(x))\\
 &\le |b^k(x)-b(x)|+|f^{k+1, \infty}(x)-b(x)|
 \\
 &\le |b^{k}(x) - b(x)| + \sum_{j = k+1}^{\infty}|b^{j}(x)-b(x)|
\end{align*}
and \eqref{claimstrong} holds. 
Therefore by the assumption \eqref{eq:fast_lqp} , $\| f^{n, k} - f^{n, \infty} \|_{L_{q}([0, 1], L_{p}(\R))} \le  2^{1-k},$ so 
 the sequence $f^{n, k}$ converges to $f^{n, \infty}$ in $L_{q}([0, 1], L_{p}(\R))$ as $k$ goes to infinity. Finally, $b - f^{n, \infty}= \sup_{j \ge n} (b - b^{j})$, so again 
\[
\lim_{n\to\infty} \| b- f^{n, \infty} \|_{L_{q}([0, 1], L_{p}(\R))} \le \lim_{n\to\infty}\sum_{j = n}^{\infty} \| b^{j} - b \|_{L_{q}([0, 1], L_{p}(\R))} \le \lim_{n\to\infty}2^{1-n}=0. 
\]

\textbf{Step 2.} Let $0 \le n \le k$. The function $f^{n,k}$  is Lipschitz and bounded. Therefore \Gref{x;f^{n, k};Y} has a unique strong solution, which we denote by $\varphi^{n, k}$. Let us provide uniform in $k,n$ bounds on the moments of $\varphi_t^{n,k}$.

By the hypothesis on $(b^{j})_{j}$, the function $f^{n, k}$ is bounded. 
Now because  $\varphi^{n, k}$ solves \Gref{x;f^{n, k};Y}, 
\begin{equation*}
	\|  [\varphi^{n, k}]_{\C^{1-\var}([0, 1])}\|_{L_{m}(\Omega)} \le \| f^{n, k} \|_{L_{\infty }([0, 1] \times \R)} < \infty, \quad m \ge 2.
\end{equation*}
We apply Corollary~\ref{cor:step02b} to $\varphi^{n, k}$ and  obtain that there exists a constant $\gamma > 0$ such that for any $m \ge 2$ and any $(s, t) \in \Delta ^{2}_{0, 1}$,
\begin{align}\label{eq:l2varphi_nk_bound}
	\|[\varphi^{n, k}]_{\C^{1-\var}([s, t])}\|_{L_{m}(\Omega)}  &\le C \bigl( \| f^{n, k} \|_{L_{q}([0, 1], L_{p}(\R))} +  \| f^{n, k} \|_{L_{q}([0, 1], L_{p}(\R))}^{q} \bigr)(t-s)^{\gamma}\nn\\
	&\le  C \bigl(1+ \| b \|_{L_{q}([0, 1], L_{p}(\R))}^{q} \bigr)(t-s)^{\gamma},
\end{align}
where the constant $C=C(H,p,q,m)$ is independent of $n, k, s$ and $t$, and the last inequality follows from \eqref{fnkbound}. Since $\varphi_0^{n,k}=0$ by definition, \eqref{eq:l2varphi_nk_bound} implies that for any $t\in[0,1]$
\begin{equation}\label{momentbound}
\|\varphi_t^{n, k}\|_{L_{m}(\Omega)} \le  C \bigl(1+ \| b \|_{L_{q}([0, 1], L_{p}(\R))}^{q} \bigr),
\end{equation}
for $C=C(H,p,q,m)$.

\textbf{Step 3.} Now we construct a sequence of solutions to regularized equations that converges almost surely to a limit.  Note that, by definition, for any $0 \le n < k$, 
\[
f^{n, k+1} \le f^{n,k} \le  f^{n+1,k}.
\]
Therefore, by a standard comparison principle, we obtain almost surely
\begin{equation}
\label{eq:many_inequalities}
\varphi^{n, k+1}_{t} \le \varphi^{n,k}_{t} \le  \varphi^{n+1,k}_{t}, 
\quad 0 \le n < k,\ t \in [0, 1].
\end{equation}
Let's define 
\[
  \varphi ^{n, \infty }_{t} := \inf_{k \ge n} \varphi ^{n, k}_{t}, \quad n \in \N,\ t \in [0, 1],
\]
where the functions may take an infinite value. 
By \eqref{eq:many_inequalities}, almost surely, for any $n \in \N$ and any $t \in [0, 1]$, the sequence $(\varphi ^{n, k}_{t})_{k}$ is decreasing, so almost surely,
\[
  \lim_{k \to \infty} \varphi _{t}^{n, k} = \varphi_{t}^{n, \infty},\quad t \in [0, 1],\ n \in \N.
\]
Therefore, recalling \eqref{momentbound} and using Fatou's lemma, we get 
\begin{equation}\label{onemorebound}
\|\varphi_t^{n, \infty}\|_{L_{m}(\Omega)} \le  C \bigl(1+ \| b \|_{L_{q}([0, 1], L_{p}(\R))}^{q} \bigr)
\end{equation}
and thus for any $t\in[0,1]$ we have $|\varphi_t^{n, \infty}|<\infty$ a.s.

We now can pass to the limit in \eqref{eq:many_inequalities} for fixed $n$ as $k\to\infty$, and get that  almost surely
\begin{equation}
\label{eq:many_inequalities_infty}
\varphi^{n,\infty}_{t} \le  \varphi^{n+1,\infty}_{t}, 
\quad 0 \le n,\ t \in [0, 1].
\end{equation}
Now we define 
\[
  \varphi _{t}^{\infty, \infty} := \sup_{n \ge 0} \varphi _{t}^{n, \infty}, \quad t \in [0, 1].
\]
By \eqref{eq:many_inequalities_infty}, almost surely, 
\begin{equation}\label{convergenceinf}
\lim_{n \to \infty} \varphi _{t}^{n, \infty} =  \varphi _{t}^{\infty, \infty},
\quad t \in [0, 1].
\end{equation}
Using again Fatou's lemma and \eqref{onemorebound}, we get 
\begin{equation*}
	\|\varphi_t^{\infty, \infty}\|_{L_{m}(\Omega)} \le  C \bigl(1+ \| b \|_{L_{q}([0, 1], L_{p}(\R))}^{q} \bigr)
\end{equation*}
 for $C=C(H,p,q,m)$. Therefore  for any $t\in[0,1]$ the random variable $\varphi_t^{\infty, \infty}$ is finite a.s.
The function $\varphi_{t}^{\infty, \infty}$ will be our candidate for a solution.

\textbf{Step 4.} Now let us show that a modification of   $\varphi^{\infty, \infty}$ solves \Gref{x;b;Y}. 

By choosing  $m > 1/\gamma$ in \eqref{eq:l2varphi_nk_bound} and applying Lemma \ref{lem:limit1var}, we get that there exists a continuous modification  $\wt \varphi^{n, \infty}$ of  $\varphi^{n, \infty}$ and a constant $C>0$ such that
\begin{equation}
\label{eq:l2varphi_n_bound}
  \|[\wt \varphi^{n, \infty}]_{\C^{1-\var}([s, t])}\|_{L_{m}(\Omega)} \le C(t-s)^{\gamma}, \quad n \in \N,\ s,t \in \Delta^{2}_{0, 1}.
\end{equation}

Fix $n \in \N$. As $k$ goes to infinity, the strong solutions $\varphi^{n, k}$ of \Gref{x;f^{n, k};Y} converge, for any $t \in [0, 1]$, almost surely to $\wt \varphi^{n, \infty}$. The drifts $f^{n, k}$  are bounded and converge to $f^{n, \infty}$ in $L_{q}([0, 1], L_{p}(\R))$ by Step 1.
Therefore, by Lemma~\ref{lem:stability}, $\wt \varphi^{n, \infty}$ solves \Gref{x;f^{n, \infty};Y} almost surely.

Since $\wt \varphi^{n, \infty}$ is a modification of $ \varphi^{n, \infty}$, it follows from \eqref{convergenceinf}, that  for any $t \in [0, 1]$, almost surely, $\lim_{n \to \infty }\wt \varphi ^{n, \infty }_{t} = \varphi ^{\infty , \infty }_{t}$. 
Therefore, recalling \eqref{eq:l2varphi_n_bound}, we can again apply Lemma~\ref{lem:stability}, this time as $n$ goes to infinity,  we obtain that $\varphi^{\infty, \infty}$ has a continuous modification $\wt \varphi^{\infty, \infty}$ and almost surely,
$\wt \varphi^{\infty, \infty}$ solves \Gref{x;b;Y}.

\textbf{Step 5.} It is easy to see, that if $\wt \varphi^{\infty, \infty}$ solves \Gref{x;b;Y} and $Y=W^H$, then $X:=x+\wt \varphi^{\infty, \infty}+W^H$ is a strong solution to \eqref{eq:main_eq}. Similarly, by taking $Y=L$, we see that $X:=x+\wt \varphi^{\infty, \infty}+L$ is a strong solution to \eqref{eq:main_eq_levy}.
\end{proof}


\appendix \section{Appendix} 
\label{sec:preliminary} 

\begin{proposition}[Representation of the fractional Brownian motion] \label{prop:representation}
	 Let  $H\in(0, 1)$. Let $W^H$ be a $d$-dimensional fractional Brownian motion. Then the following holds:
\begin{enumerate}[\rm{(}i\rm{)}]
	\item there exists a measurable function $\Phi\colon C ([0, 1], \R^d) \to C ([0, 1], \R^d)$ such that $B := \Phi (W^H)$ is a Brownian motion;
	\item one has 
	\begin{equation}
		\label{WB} W^H_t=\int_0^t K_H(t,s)\,dB_s =: \Psi_t(B),\quad t\in[0,T],
	\end{equation}
	where the kernel $K$ is defined in~\eqref{fbmr1}--\eqref{fbmr2};
	\item  $\Psi \circ \Phi (W^H) = W^H$;
	\item if  $H \le \frac{1}{2}$ then $\Phi : C ([0, 1], \R^d) \to C ([0, 1], \R^d)$ is continuous;
	\item if $H \ge \frac{1}{2}$, $\theta>H-\frac12$, then $\Psi : \C_0^{\theta}([0, 1], \R^d) \to C ([0, 1], \R^d) $ is continuous;
	\item \label{part1fbm} for $(s,t)\in\Delta_{0,T}^2$ one has $K_H(t, s)\ge C(H,d)(t-s)^{H-\frac12}$.
\end{enumerate}
\end{proposition}
\begin{proof}
(i) The representation (\ref{WB}) is proved in \cite[Section~5]{Nu06}. 
(ii), (iii) is \cite[Theorem~11]{Picard}. (iv) is \cite[Lemma~B.1]{Art}. 
(v) The continuity of the map  $\Psi$ is covered in \cite[Proposition B.3(iv)]{BLM23}. (vi) This is obvious and can be found, e.g., in  \cite[Proposition B.2(ii)]{BLM23}.
\end{proof}

\begin{proposition}[{John--Nirenberg's inequality, \cite[Theorem 2.3]{Le22}, \cite[Theorem~I.6.10]{Bass}}]
  \label{prop:John-Nierenberg}
  Let $0\le S\le T$. Let $\A$ be a continuous stochastic process $\Omega\times[S,T]\to\R$ adapted to a filtration $(\F_t)_{t\in[S,T]}$.
  Assume that for some constant $\Gamma>0$ one has 
  \[ 
  \E^s |\A_t-\A_s|\le \Gamma,\quad  \text{for any }(s, t) \in \Delta^2_{S, T}.
  \]
  Then for all integer $n \ge 1$ there exists a constant $C = C(n)$ such that
  \[
    \E [|\A_T-\A_S|^n] \le C \Gamma^n.
  \]
\end{proposition}

\begin{proposition}[{Taming singularities lemma  \cite[Lemma~3.4]{le2021taming}}]\label{prop:taming_singularities}

  Let $(\mathcal E,d)$ be a metric space.
  Suppose that there exist constants $\tau_1,\eta_1,\tau_2,\eta_2\ge0$, $\tau_1>\eta_1$, $\tau_2>\eta_2$, $\Gamma_1, \Gamma_2>0$, such that a continuous function $Y\colon (0,T]\to \mathcal E$ satisfies
  \begin{equation}
    \label{c:Ysing} d(Y_s,Y_t) \le \Gamma_1 s^{-\eta_1}(t-s)^{\tau_1}+\Gamma_2 s^{-\eta_2}(t-s)^{\tau_2}\quad\text{for any $0< s\le t\le T$}.
  \end{equation}
  Then for some universal constant $C=C(\tau_1,\eta_1,\tau_2,\eta_2,T)$ one has
  \begin{equation*}
    d(Y_s,Y_t)\le C \Gamma_1(t-s)^{\tau_1- \eta_1}+\Gamma_2(t-s)^{\tau_2- \eta_2}\quad\text{for any $0\le s\le t\le T$}.
  \end{equation*}
\end{proposition}


\begin{lemma}
  \label{lem:sqrbound}
  Let $x \ge 0$, $\Gamma \ge 0$ and $0 \le \gamma <1$ be such that
  \[
    x \le \Gamma (1 + x^{\gamma}).
  \]
  Then there exists $C=C(\gamma)$ such that
  \[
    x \le C(\Gamma + \Gamma^{\frac{1}{1-\gamma}}).
  \]
\end{lemma}
\begin{proof}
  If  $x \le 1$, then $x \le 2 \Gamma$.   If alternatively, $x > 1$, then $x^{\gamma} \ge 1$ and therefore $x \le 2 \Gamma x^{\gamma}$, so $x \le (2 \Gamma)^{\frac{1}{1-\gamma}}$.
  Thus, we can conclude with the constant $C = 2^{\frac{1}{1-\gamma}}$.
\end{proof}

\begin{proposition}
\label{prop:skorokhod} 
Let $(L^n)_{n\in\N}$ be a sequence of Lévy processes that are identically distributed. Assume that $L^n$ converges almost surely to $L$ in $\D ([0, 1], \R^d)$. 
Then, for any fixed $t\in [0, 1]$ the sequence $L_{t}^n$ converges to $L_{t}$ almost surely as $n\to\infty$, and $L$ is a Lévy process with the same law as any of the $L^{n}$.
\end{proposition}
\begin{proof}
Fix $t \in [0, 1]$. A Lévy process is stochastically continuous by one of the equivalent definitions (see for example \cite[Section 1.3]{App05}). Therefore, almost surely, $L$ is continuous at $t$. By \cite[Lemma~VI.2.3]{JS87}, this proves that almost surely, the map $\Lambda_t(x):=x(t)$, $x\in\D([0, 1])$, is continuous at $x=L$. Therefore, since $L^n$ converges to $L$ in $\D([0, 1])$ almost surely, we have that $L^n_t=\Lambda_t(L^n)$ converges almost surely to $\Lambda_t(L)=L_t$.
This shows that the finite dimensional laws of $L$ are the same as that of $L^{n}$ for any $n \ge 0$.
Finally, because $L$ is càdlàg, its finite dimesional laws characterize its law as a process. Therefore $L$ is a Lévy process with the same law as $L^{n}$.
\end{proof}

\begin{lemma}[The freezing lemma, {\cite[Lemma 4.1]{Baldi}}]
\label{lem:freezing}
Let $(\Omega, \F, \P)$ be a probability space and $\mathcal G$ and $\mathcal H$ independent sub-$\sigma$-algebras of $\mathcal F$. Let $X$ be a $\mathcal H$-measurable random variable taking value in the measurable space $(E, \mathcal E)$ and $\psi : E \times \Omega \to \R$ an $\mathcal E \otimes \mathcal G$-measurable function such that $\omega \mapsto \psi(X(\omega), \omega)$ is integrable. Then
\[
E[\psi(X, \cdot) | \mathcal H] = \phi(X),
\]
where $\phi(x) = \E[\psi(x, \cdot)]$. 
\end{lemma}

\begin{lemma}
\label{lem:limit1var}
Let $\psi^{n} : [0, 1] \to \R^{d}$, $n \in \N$, be a sequence of continuous stochastic processes such that for some $m \ge 1$ and $\gamma > 1/m$, there exists a constant $C>0$ such that 
\begin{equation}
\label{eq:lim1var}
  \| [\psi ^{n}]_{\C^{1-\var}([s, t])} \|_{L_{m}(\Omega)} \le C (t-s)^{\gamma}, \quad (s, t) \in \Delta^{2}_{0,1},\ n \in \N.
\end{equation}
Assume moreover that there exists a stochastic process $\psi : [0, 1] \to \R^{d}$ such that for any $t \in [0, 1]$, $\lim_{n \to \infty} \psi^{n}_{t} = \psi_{t}$ almost surely.

Then there exists a continuous modification of $\psi$. Furthermore, any continuous modification $\wt \psi$ of $\psi$ satisfies 
 \begin{equation}
\label{eq:result_lim1var}
  \| [\wt \psi]_{\C^{1-\var}([s, t])} \|_{L_{m}(\Omega)} \le C (t-s)^{\gamma}, \quad (s, t) \in \Delta^{2}_{0,1},
\end{equation}
and for any $t \in [0, 1]$, $\lim_{n \to \infty} \psi^{n}_{t} = \wt \psi_{t}$ almost surely.
\end{lemma}

\begin{proof}
By \eqref{eq:lim1var},
\begin{equation*}
	\| \psi ^{n}_{t} - \psi ^{n}_{s} \|_{L_{m}(\Omega)} \le C(t-s)^{\gamma},\quad (s, t) \in \Delta^{2}_{0, 1},\ n \in \N,
\end{equation*}
and applying Fatou's lemma we obtain
\begin{equation}
  \label{eq:lem_kolmogorov_bound}
	\| \psi _{t} - \psi_{s} \|_{L_{m}(\Omega)} \le C(t-s)^{\gamma},\quad (s, t) \in \Delta^{2}_{0, 1}.
\end{equation}
Because $\gamma > 1/m$ by hypothesis, we can apply Kolmogorov's continuity criterion and therefore there is a continuous modification of $\psi$. 

We now show \eqref{eq:result_lim1var}. Fix $(s, t) \in \Delta ^{2}_{0, 1}$ and let $\wt \psi$ be a continuous version of $\psi$. 
Using the convergence hypothesis, for any $r \in [0, 1]$, we still have $ \lim_{n \to \infty} \psi^{n}_{r} = \wt \psi _{r}$ almost surely. Because the rational numbers are countable  
\begin{equation*}
\P\bigl(\text{for any $r = s + q(t-s)$, $q \in \mathbb Q$ we have} \lim_{n \to \infty} \psi^{n}_{r} = \wt \psi _{r}\bigr)=1.
\end{equation*}
Now using that $\wt \psi$ continuous, we obtain almost surely
\begin{align*}
  [\wt \psi]_{\C^{1-\var}([s, t])} 
  &= \sup_{N \in \N} \sum_{i = 0}^{N-1} \Big|\wt \psi \Bigl(s + \frac{i+1}{N}(t-s)\Bigr) -\wt \psi\Bigr(s + \frac{i}{N}(t-s)\Bigr)\Big| 
\\&= \sup_{N\in \N}  \liminf_{n \to \infty} \sum_{i = 0}^{N-1} \Big |\psi^{n}\Bigl(s + \frac{i+1}{N}(t-s)\Bigr ) -\psi^{n}\Bigl (s + \frac{i}{N}(t-s) \Bigr )\Big | 
\\&\le \liminf_{n \to \infty}  [\psi^{n}]_{\C^{1-\var}([s, t])}.
\end{align*}
Finally, by taking the $L_{m}(\Omega)$ norm in the previous inequality and by applying Fatou's lemma, we showed that we can legitimately pass to the limit in \eqref{eq:lim1var} and we obtain
\begin{equation*}
	\| [\wt \psi]_{\C^{1-\var}([s, t])} \|_{L_{m}(\Omega)} \le C(t-s)^{\gamma}.
\end{equation*}
The constant in the above inequality is independent of $(s, t)\in \Delta ^{2}_{0,1}$, and the choice of $s$ and $t$ was arbitrary. Therefore \eqref{eq:result_lim1var} holds.
\end{proof}

\bibliographystyle{amsplain}
\bibliography{biblio.bib}

@article{Art,
  title = {Regularisation by Fractional Noise for One-Dimensional Differential Equations with Distributional Drift},
  author = {Anzeletti, Lukas and Richard, Alexandre and Tanr{\'e}, Etienne},
  year = {2023},
  journal = {Electronic Journal of Probability},
  volume = {28},
  number = {135},
  pages = {49},
  issn = {1083-6489},
  doi = {10.1214/23-ejp1010},
  mrnumber = {4664455}
}

@book{App05,
  title = {L{\'e}vy Processes and Stochastic Calculus},
  author = {Applebaum, David},
  year = {2005},
  series = {Cambridge Studies in Advanced Mathematics},
  edition = {Reprint},
  number = {93},
  publisher = {Cambridge University Press},
  address = {Cambridge},
  isbn = {978-0-521-83263-2},
  langid = {english},
  mrnumber = {2072890}
}

@article {ABM,
    AUTHOR = {Athreya, Siva and Butkovsky, Oleg and Mytnik, Leonid},
     TITLE = {Strong existence and uniqueness for stable stochastic
              differential equations with distributional drift},
   JOURNAL = {Annals of Probability},
  FJOURNAL = {The Annals of Probability},
    VOLUME = {48},
      YEAR = {2020},
    NUMBER = {1},
     PAGES = {178--210},
      ISSN = {0091-1798,2168-894X},
   MRCLASS = {60H10 (60G52 60H50)},
  MRNUMBER = {4079434},
       DOI = {10.1214/19-AOP1358},
       URL = {https://doi.org/10.1214/19-AOP1358},
}

@book{Baldi,
  title = {Stochastic {{Calculus}}},
  author = {Baldi, Paolo},
  year = {2017},
  series = {Universitext},
  publisher = {Springer International Publishing},
  address = {Cham},
  doi = {10.1007/978-3-319-62226-2},
  urldate = {2025-07-16},
  copyright = {http://www.springer.com/tdm},
  isbn = {978-3-319-62225-5 978-3-319-62226-2},
  mrnumber = {3726894}
}

@book{Bass,
  title = {Probabilistic Techniques in Analysis},
  author = {Bass, Richard F.},
  year = {1995},
  series = {Probability and Its {{Applications}} ({{New York}})},
  publisher = {Springer-Verlag, New York},
  isbn = {978-0-387-94387-9},
  mrnumber = {1329542}
}

@article{BDG24,
  title = {Strong Rate of Convergence of the {{Euler}} Scheme for {{SDEs}} with Irregular Drift Driven by {{Levy}} Noise},
  author = {Butkovsky, Oleg and Dareiotis, Konstantinos and Gerencs{\'e}r, M{\'a}t{\'e}},
  year = {2024},
  journal = {arXiv preprint arXiv:2204.12926},
  doi = {10.48550/arXiv.2204.12926}
  }

@article{BLM23,
  title = {Stochastic Equations with Singular Drift Driven by Fractional {{Brownian}} Motion},
  author = {Butkovsky, Oleg and L{\^e}, Khoa and Mytnik, Leonid},
  year = {2025},
  journal = {Probability and Mathematical Physics},
  volume = {6},
  number = {3},
  pages = {857--912},
  issn = {2690-0998,2690-1005},
  doi = {10.2140/pmp.2025.6.857},
  mrnumber = {4930606}
}

@article{CdRM22,
  title = {On Multidimensional Stable-Driven Stochastic Differential Equations with {{Besov}} Drift},
  author = {{Chaudru de Raynal}, Paul-{\'E}ric and Menozzi, St{\'e}phane},
  year = {2022},
  journal = {Electronic Journal of Probability},
  volume = {27},
  number = {163},
  pages = {52},
  issn = {1083-6489},
  doi = {10.1214/22-ejp864},
  mrnumber = {4525442}
}

@article{CG,
  title = {Averaging along Irregular Curves and Regularisation of {{ODEs}}},
  author = {Catellier, R{\'e}mi and Gubinelli, Massimiliano},
  year = {2016},
  journal = {Stochastic Processes and their Applications},
  volume = {126},
  number = {8},
  pages = {2323--2366},
  issn = {0304-4149,1879-209X},
  doi = {10.1016/j.spa.2016.02.002},
  urldate = {2025-07-15},
  mrnumber = {3505229}
}

@article{CZZ,
  title = {Supercritical {{SDEs}} Driven by Multiplicative Stable-like {{L{\'e}vy}} Processes},
  author = {Chen, Zhen-Qing and Zhang, Xicheng and Zhao, Guohuan},
  year = {2021},
  journal = {Transactions of the American Mathematical Society},
  volume = {374},
  number = {11},
  pages = {7621--7655},
  issn = {0002-9947, 1088-6850},
  doi = {10.1090/tran/8343},
  urldate = {2025-07-16},
  langid = {english},
  mrnumber = {4328678}
}

@book{FV10,
  title = {Multidimensional Stochastic Processes as Rough Paths: Theory and Applications},
  shorttitle = {Multidimensional Stochastic Processes as Rough Paths},
  author = {Friz, Peter K. and Victoir, Nicolas B.},
  year = {2010},
  series = {Cambridge {{Studies}} in {{Advanced Mathematics}}},
  number = {120},
  publisher = {Cambridge University Press},
  address = {Cambridge, UK New York},
  doi = {10.1017/CBO9780511845079},
  isbn = {978-0-511-84507-9 978-0-511-67754-0 978-0-511-68004-5},
  langid = {english},
  mrnumber = {2604669}
}

@article{galeati2023note,
    AUTHOR = {Galeati, Lucio},
     TITLE = {A note on weak existence for singular {SDE}s},
   JOURNAL = {Stoch. Dyn.},
  FJOURNAL = {Stochastics and Dynamics},
    VOLUME = {24},
      YEAR = {2024},
    NUMBER = {3},
     PAGES = {Paper No. 2450025, 18},
      ISSN = {0219-4937,1793-6799},
   MRCLASS = {60H10 (60H50)},
  MRNUMBER = {4785017},
MRREVIEWER = {Paulo\ R. C. Ruffino},
       DOI = {10.1142/S0219493724500254},
       URL = {https://doi.org/10.1142/S0219493724500254},
}

@article{GG23,
  title = {Solution Theory of Fractional {{SDEs}} in Complete Subcritical Regimes},
  author = {Galeati, Lucio and Gerencs{\'e}r, M{\'a}t{\'e}},
  year = {2025},
  journal = {Forum of Mathematics, Sigma},
  volume = {13},
  pages = {66},
  issn = {2050-5094},
  doi = {10.1017/fms.2024.136},
  urldate = {2025-07-15},
  langid = {english},
  mrnumber = {4854429}
}

@article{GP93a,
  title = {On Quasi-Linear Stochastic Partial Differential Equations},
  author = {Gy{\"o}ngy, Istv{\'a}n and Pardoux, {\'E}tienne},
  year = {1993},
  journal = {Probability Theory and Related Fields},
  volume = {94},
  number = {4},
  pages = {413--425},
  issn = {1432-2064},
  doi = {10.1007/BF01192556},
  urldate = {2025-07-16},
  langid = {english},
  mrnumber = {1201552}
}

@article{GP93b,
  title = {On the Regularization Effect of Space-Time White Noise on Quasi-Linear Parabolic Partial Differential Equations},
  author = {Gy{\"o}ngy, Istv{\'a}n and Pardoux, {\'E}tienne},
  year = {1993},
  journal = {Probability Theory and Related Fields},
  volume = {97},
  number = {1},
  pages = {211--229},
  issn = {0178-8051,1432-2064},
  doi = {10.1007/BF01199321},
  urldate = {2025-07-16},
  langid = {english},
  mrnumber = {1240724}
}

@article{hao2023sdes,
  title = {{{SDEs}} with Supercritical Distributional Drifts},
  author = {Hao, Zimo and Zhang, Xicheng},
  year = {2024},
  journal = {arXiv preprint arXiv:2312.11145},
  doi = {10.48550/arXiv.2312.11145}
  }

@article{Jin18,
  title = {On Weak Solutions of {{SDEs}} with Singular Time-Dependent Drift and Driven by Stable Processes},
  author = {Jin, Peng},
  year = {2018},
  journal = {Stochastics and Dynamics},
  volume = {18},
  number = {02},
  pages = {23},
  issn = {0219-4937,1793-6799},
  doi = {10.1142/S0219493718500132},
  urldate = {2025-07-16},
  mrnumber = {3735412}
}

@book{JS87,
  title = {Limit Theorems for Stochastic Processes},
  author = {Jacod, Jean and Shiryaev, Albert N.},
  year = {1987},
  series = {Grundlehren Der Mathematischen {{Wissenschaften}}},
  volume = {288},
  publisher = {Springer-Verlag, Berlin},
  doi = {10.1007/978-3-662-02514-7},
  isbn = {978-3-540-17882-8},
  mrnumber = {959133}
}

@article{Knopova,
  title = {A Note on the Existence of Transition Probability Densities of {{L{\'e}vy}} Processes},
  author = {Knopova, Victoria and Schilling, Ren{\'e} L.},
  year = {2013},
  journal = {Forum Mathematicum},
  volume = {25},
  number = {1},
  pages = {125--149},
  issn = {0933-7741,1435-5337},
  doi = {10.1515/form.2011.108},
  urldate = {2025-07-15},
  copyright = {De Gruyter expressly reserves the right to use all content for commercial text and data mining within the meaning of Section 44b of the German Copyright Act.},
  langid = {english},
  mrnumber = {3010850}
}

@article {KP22,
    AUTHOR = {Kremp, Helena and Perkowski, Nicolas},
     TITLE = {Multidimensional {SDE} with distributional drift and {L}\'evy
              noise},
   JOURNAL = {Bernoulli},
  FJOURNAL = {Bernoulli. Official Journal of the Bernoulli Society for
              Mathematical Statistics and Probability},
    VOLUME = {28},
      YEAR = {2022},
    NUMBER = {3},
     PAGES = {1757--1783},
      ISSN = {1350-7265,1573-9759},
   MRCLASS = {60H10 (60G51)},
  MRNUMBER = {4411510},
       DOI = {10.3150/21-bej1394},
       URL = {https://doi.org/10.3150/21-bej1394},
}

@article{KP25rough,
  title = {Rough Weak Solutions for Singular {{L{\'e}vy SDEs}}},
  author = {Kremp, Helena and Perkowski, Nicolas},
  year = {2023},
  journal = {arXiv preprint arXiv:2309.15460},
  doi = {10.1007/s00440-025-01371-y},
}

@article{KR05,
  title = {Strong Solutions of Stochastic Equations with Singular Time Dependent Drift},
  author = {Krylov, Nikolay V. and R{\"o}ckner, Michael},
  year = {2005},
  journal = {Probability Theory and Related Fields},
  volume = {131},
  number = {2},
  pages = {154--196},
  issn = {0178-8051,1432-2064},
  doi = {10.1007/s00440-004-0361-z},
  urldate = {2025-07-15},
  langid = {english},
  mrnumber = {2117951}
}

@article{Kr20,
  title = {On Time Inhomogeneous Stochastic {{It{\^o}}} Equations with Drift in {{L}}{\textsubscript{d+1}}},
  author = {Krylov, Nikolay V.},
  year = {2021},
  journal = {Ukrainian Mathematical Journal},
  volume = {72},
  number = {9},
  pages = {1420--1444},
  issn = {1573-9376},
  doi = {10.1007/s11253-021-01864-8},
  urldate = {2025-07-15},
  langid = {english},
  mrnumber = {4207066}
}

@article{KS19,
  title = {Strong Convergence of the {{Euler}}--{{Maruyama}} Approximation for a Class of {{L{\'e}vy-driven SDEs}}},
  author = {K{\"u}hn, Franziska and Schilling, Ren{\'e} L.},
  year = {2019},
  journal = {Stochastic Processes and their Applications},
  volume = {129},
  number = {8},
  pages = {2654--2680},
  issn = {0304-4149,1879-209X},
  doi = {10.1016/j.spa.2018.07.018},
  urldate = {2025-07-16},
  mrnumber = {3980140}
}

@article{KS21,
  title = {{Convolution inequalities for Besov and Triebel--Lizorkin spaces, and applications to convolution semigroups}},
  author = {K{\"u}hn, Franziska and Schilling, Ren{\'e} L.},
  year = {2022},
  journal = {Studia Mathematica},
  volume = {262},
  pages = {93--119},
  issn = {0039-3223, 1730-6337},
  doi = {10.4064/sm210127-23-3},
  urldate = {2025-07-15},
  langid = {polish},
  mrnumber = {4339468}
}

@article{le2021taming,
  title = {Taming Singular Stochastic Differential Equations: {{A}} Numerical Method},
  shorttitle = {Taming Singular Stochastic Differential Equations},
  author = {L{\^e}, Khoa and Ling, Chengcheng},
  year = {2025},
  journal = {arXiv preprint arXiv:2110.01343},
  doi = {10.48550/arXiv.2110.01343}
  }

@article{Le22,
  title = {Quantitative {{John--Nirenberg}} Inequality for Stochastic Processes of Bounded Mean Oscillation},
  author = {L{\^e}, Khoa},
  year = {2022},
   journal = {arXiv preprint arXiv:2210.15736},
  doi = {10.48550/arXiv.2210.15736}
  }

@article{Le22b,
  title = {Maximal inequalities and weighted {B}{M}{O} processes},
  author = {L{\^e}, Khoa},
  year = {2022},
  journal = {arXiv preprint arXiv:2211.15550},
  doi = {10.48550/arXiv2211.15550}
  }

@article{LeBanach,
  title = {Stochastic Sewing in {{Banach}} Spaces},
  author = {L{\^e}, Khoa},
  year = {2023},
  journal = {Electronic Journal of Probability},
  volume = {28},
  pages = {1--22},
  issn = {1083-6489, 1083-6489},
  doi = {10.1214/23-EJP918},
  urldate = {2025-07-15},
  mrnumber = {4546635}
}

@article{LeSSL,
  title = {A Stochastic Sewing Lemma and Applications},
  author = {L{\^e}, Khoa},
  year = {2020},
  journal = {Electronic Journal of Probability},
  volume = {25},
  number = {38},
  pages = {55},
  issn = {1083-6489, 1083-6489},
  doi = {10.1214/20-EJP442},
  urldate = {2025-07-15},
  mrnumber = {4089788}
}

@article{LZ22,
  title = {Nonlocal Elliptic Equation in {{H{\"o}lder}} Space and the Martingale Problem},
  author = {Ling, Chengcheng and Zhao, Guohuan},
  year = {2022},
  journal = {Journal of Differential Equations},
  volume = {314},
  pages = {653--699},
  issn = {0022-0396,1090-2732},
  doi = {10.1016/j.jde.2022.01.025},
  urldate = {2025-07-16},
  mrnumber = {4369182}
}

@article{MW24,
  title = {Strong {{Existence}} and {{Uniqueness}} for {{Singular SDEs Driven}} by {{Stable Processes}}},
  author = {Mytnik, Leonid and Weinberger, Johanna},
  year = {2024},
  journal = {arXiv preprint arXiv:2404.13729},
  doi = {10.48550/arXiv.2404.13729}
  }

@book{Nu06,
  title = {The {{Malliavin Calculus}} and {{Related Topics}}},
  author = {Nualart, David},
  year = {2006},
  series = {Probability, Its {{Applications}}},
  edition = {Second},
  publisher = {Springer-Verlag},
  address = {Berlin/Heidelberg},
  urldate = {2025-07-15},
  copyright = {http://www.springer.com/tdm},
  isbn = {978-3-540-28328-7},
  langid = {english},
  mrnumber = {2200233}
}

@incollection{Picard,
  title = {Representation {{Formulae}} for the {{Fractional Brownian Motion}}},
  booktitle = {S{\'e}minaire de {{Probabilit{\'e}s XLIII}}},
  author = {Picard, Jean},
  editor = {{Donati-Martin}, Catherine and Lejay, Antoine and Rouault, Alain},
  year = {2011},
  series = {Lecture {{Notes}} in {{Math}}.},
  pages = {3--70},
  publisher = {Springer},
  address = {Berlin, Heidelberg},
  doi = {10.1007/978-3-642-15217-7_1},
  urldate = {2025-07-15},
  isbn = {978-3-642-15217-7},
  langid = {english},
  mrnumber = {2790367}
}

@article{Portenko,
  title = {On Multidimensional Stable Processes with Locally Unbounded Drift},
  author = {Podolynny, S. I. and Portenko, N. I.},
  year = {1995},
  journal = {Random Operators and Stochastic Equations},
  volume = {3},
  number = {2},
  pages = {113--124},
  issn = {0926-6364,1569-397X},
  doi = {10.1515/rose.1995.3.2.113},
  urldate = {2025-07-15},
  copyright = {De Gruyter expressly reserves the right to use all content for commercial text and data mining within the meaning of Section 44b of the German Copyright Act.},
  langid = {english},
  mrnumber = {1341116}
}

@article{Sasha,
  title = {On the {{Uniqueness}} in {{Law}} and the {{Pathwise Uniqueness}} for {{Stochastic Differential Equations}}},
  author = {Cherny, Alexander S.},
  year = {2002},
  journal = {Theory of Probability \& Its Applications},
  volume = {46},
  number = {3},
  pages = {406--419},
  publisher = {{Society for Industrial and Applied Mathematics}},
  issn = {0040-585X},
  doi = {10.1137/S0040585X97979093},
  urldate = {2025-07-15}
}

@book{Saw18,
  title = {Theory of {{Besov Spaces}}},
  author = {Sawano, Yoshihiro},
  year = {2018},
  series = {Developments in {{Mathematics}}},
  volume = {56},
  publisher = {Springer},
  address = {Singapore},
  doi = {10.1007/978-981-13-0836-9},
  urldate = {2025-07-16},
  copyright = {http://www.springer.com/tdm},
  isbn = {978-981-13-0835-2 978-981-13-0836-9},
  mrnumber = {3839617}
}

@article{Ver80,
  title = {On Strong Solutions and Explicit Formulas for Solutions of Stochastic Integral Equations},
  author = {Veretennikov, A. J.},
  year = {1981},
  journal = {Mathematics of the USSR-Sbornik},
  volume = {39},
  number = {3},
  pages = {387--403},
  issn = {0025-5734},
  doi = {10.1070/sm1981v039n03abeh001522},
  urldate = {2025-07-15},
  mrnumber = {568986}
}

@article{Zhang11,
  title = {Stochastic {{Homeomorphism Flows}} of {{SDEs}} with {{Singular Drifts}} and {{Sobolev Diffusion Coefficients}}},
  author = {Zhang, Xicheng},
  year = {2011},
  journal = {Electronic Journal of Probability},
  volume = {16},
  pages = {1096--1116},
  issn = {1083-6489, 1083-6489},
  doi = {10.1214/EJP.v16-887},
  urldate = {2025-07-15},
  mrnumber = {2820071}
}

@article{Zhang13,
  title = {Stochastic Differential Equations with {{Sobolev}} Drifts and Driven by {$\alpha$}-Stable Processes},
  author = {Zhang, Xicheng},
  year = {2013},
  journal = {Annales de l'Institut Henri Poincar{\'e}, Probabilit{\'e}s et Statistiques},
  volume = {49},
  number = {4},
  pages = {1057--1079},
  issn = {0246-0203,1778-7017},
  doi = {10.1214/12-AIHP476},
  urldate = {2025-07-16},
  mrnumber = {3127913}
}

@article{Zvo74,
  title = {A Transformation of the Phase Space of a Diffusion Process That Will Remove the Drift},
  author = {Zvonkin, Alexander},
  year = {1974},
  journal = {Mathematics of the USSR-Sbornik},
  volume = {93(135)},
  pages = {129--149},
  issn = {0368-8666},
  mrnumber = {336813}
}

@article{ZZ,
  title = {Stochastic {{Lagrangian Path}} for {{Leray}}'s {{Solutions}} of {{3D Navier}}--{{Stokes Equations}}},
  author = {Zhang, Xicheng and Zhao, Guohuan},
  year = {2021},
  journal = {Communications in Mathematical Physics},
  volume = {381},
  number = {2},
  pages = {491--525},
  issn = {0010-3616,1432-0916},
  doi = {10.1007/s00220-020-03888-w},
  urldate = {2025-07-15},
  langid = {english},
  mrnumber = {4207449}
}




\end{document}